\documentclass[11pt]{amsart}

\usepackage{amssymb,amsmath,accents}
\usepackage{bbm}
\usepackage{graphicx}
\usepackage{a4wide}

\newtheorem{theorem}{Theorem}
\newtheorem{prop}{Proposition}
\newtheorem{lemma}{Lemma}     
\newtheorem{coro}{Corollary}
\theoremstyle{definition}
\newtheorem{remark}{Remark}

\newcommand{\ts}{\hspace{0.5pt}}
\newcommand{\nts}{\hspace{-0.5pt}}
\newcommand{\AAA}{\mathbb{A}}
\newcommand{\GG}{\mathbb{G}}
\newcommand{\GGc}{\mathbb{G}^{\mathsf{c}}_{\phantom{I}}}
\newcommand{\RR}{\mathbb{R}\ts}
\newcommand{\HH}{\mathbb{H}}
\newcommand{\NN}{\mathbb{N}}
\newcommand{\PP}{\mathbb{P}}

\newcommand{\cA}{\mathcal{A}}
\newcommand{\cB}{\mathcal{B}}
\newcommand{\cC}{\mathcal{C}}
\newcommand{\cD}{\mathcal{D}}
\newcommand{\cE}{\mathcal{E}}

\newcommand{\cM}{\mathcal{M}}
\newcommand{\cP}{\mathcal{P}}
\newcommand{\cU}{\mathcal{U}}
\newcommand{\pa}{\hphantom{g}\nts\nts}
\newcommand{\vph}{\vphantom{I}}

\newcommand{\dd}{\,\mathrm{d}}
\newcommand{\ee}{\ts\mathrm{e}}

\newcommand{\rtot}[1]{\varrho^{#1}_{\mathrm{tot}}}
\newcommand{\alin}{a^{\mathrm{lin}}_{t}}
\newcommand{\pmin}{\ts\ts\underline{\nts\nts 0\nts\nts}\ts\ts}
\newcommand{\pmax}{\ts\ts\underline{\nts\nts 1\nts\nts}\ts\ts}
\newcommand{\bigtim}{\mbox{\LARGE $\times$}}
\newcommand{\udo}[1]{\underaccent{$\text{.}$}{#1\ts}\nts}
\newcommand{\exend}{\hfill$\Diamond$}

\begin{document}

\title[Solution of the recombination equation]
{The general recombination equation \\[2mm]
in continuous time and its solution}

\author{Ellen Baake}
\address{Technische Fakult\"at, Universit\"at Bielefeld, 
         Postfach 100131, 33501 Bielefeld, Germany}

\author{Michael Baake}
\address{Fakult\"at f\"ur Mathematik, Universit\"at Bielefeld, 
         Postfach 100131, 33501 Bielefeld, Germany}

\author{Majid Salamat}
\address{Technische Fakult\"at, Universit\"at Bielefeld, 
         Postfach 100131, 33501 Bielefeld, Germany}

\begin{abstract} 
  The process of recombination in population genetics, in its
  deterministic limit, leads to a nonlinear ODE in the Banach space of
  finite measures on a locally compact product space. It has an
  embedding into a larger family of nonlinear ODEs that permits a
  systematic analysis with lattice-theoretic methods for general
  partitions of finite sets. We discuss this type of system, reduce it
  to an equivalent finite-dimensional nonlinear problem,
  and establish a connection with an ancestral partitioning process,
  backward in time. We solve the
  finite-dimensional problem recursively for generic sets of parameters
  and briefly
  discuss the singular cases, and how to extend the solution to this
  situation.
\end{abstract}

\maketitle

\section{Introduction}
This contribution is concerned with differential equation models for
the dynamics of the genetic composition of populations that evolve
under recombination. Here, recombination is the genetic mechanism in
which two parent individuals are involved in creating the mixed type
of their offspring during sexual reproduction.  The essence of this
process is illustrated in Fig.~\ref{fig:lifecycle} and may be
idealised and summarised as follows.

Genetic information is encoded in terms of finite sequences.
Eggs and sperm (i.e., female and male germ cells or
\emph{gametes}) each carry one such sequence. They go through the
following life cycle: At fertilisation, two gametes meet randomly and
unite, thus starting the life of a new individual, which is equipped
with both the maternal and the paternal sequence. At maturity, this
individual generates its own germ cells. This process may include
recombination, that is, the maternal and paternal sequences perform
one or more \emph{crossovers} and are cut and relinked accordingly, so
that two `mixed' sequences emerge. These are the new gametes and start
the next round of fertilisation (by random mating within a large
population).

\begin{figure}[ht]
  \begin{center}
  \includegraphics[width=0.83\textwidth]{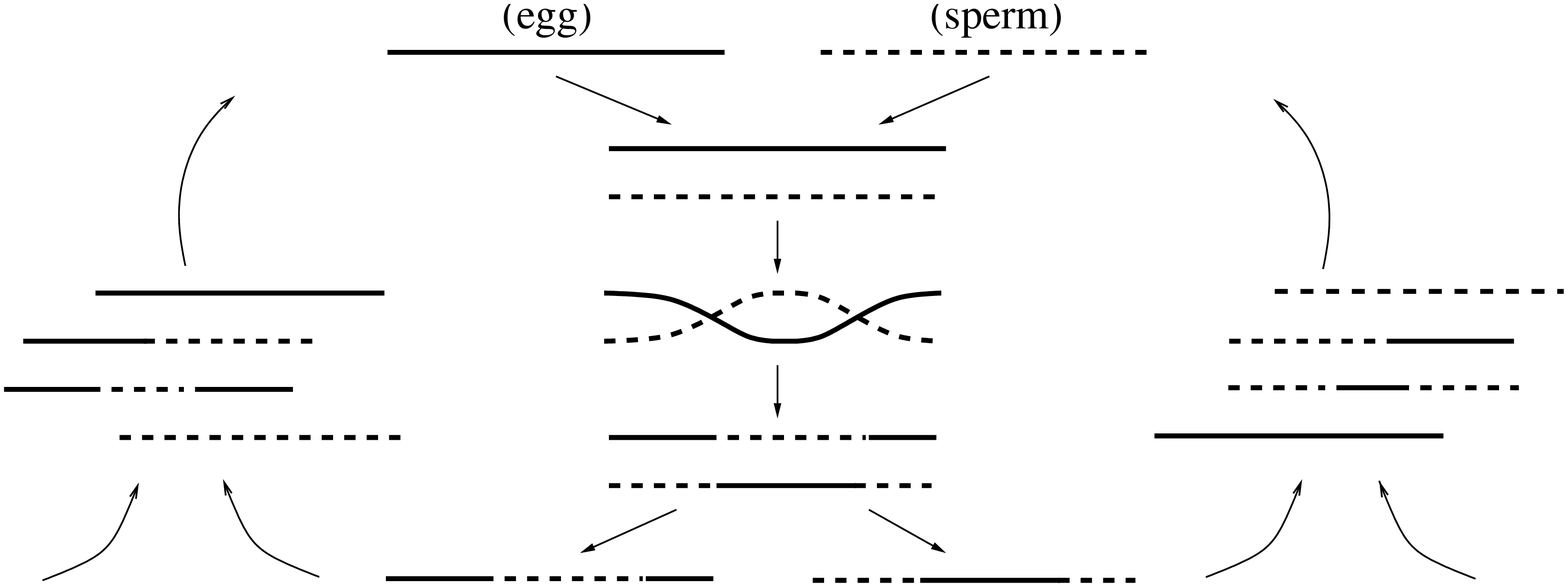}
\end{center} \caption{\label{fig:lifecycle} {\small Life cycle of a
    population under sexual reproduction and recombination. Each line
    symbolises a sequence of sites that defines a gamete (such as the
    two at the top that start the cycle as `egg' and `sperm'). The
    pool of gametes at the left and the right comes from a large
    population of recombining individuals. These sequences meet
    randomly to start the life of a new individual. Altogether, each
    of these sequences has been pieced together from two randomly
    chosen parental sequences.}}
\end{figure}

Models of this process aim at describing the dynamics of the genetic
composition of a population that goes through this life cycle
repeatedly.  These models come in various flavours: in discrete or
continuous time; with various assumptions about the crossover pattern;
and in a deterministic or a stochastic formulation, depending on
whether or not the population is assumed to be so large that
stochastic fluctuations may be neglected.  We will employ the
deterministic continuous-time approach here, but allow for very
general crossover patterns. 
The biologically relevant cases will be mentioned throughout the
paper, as will various connections to the existing body of literature.

From now on, we describe populations at the level of their gametes and
thus identify gametes with individuals.  Their genetic information is
encoded in terms of a linear arrangement of sites, indexed by the set
$S:= \{ 1, \dots , n\}$.  For each site $i\in S$, there is a set $X_i$
of `letters' that may possibly occur at that site.  For the sake of
concreteness, we use \emph{finite} sets $X_i$ for
the moment; we generalise this to arbitrary locally compact spaces
$X_i$ in Section~\ref{sec:partitioning}.

A \emph{type} is thus defined as a sequence $x=(x^{}_1, \dots ,
x^{}_n) \in X_1 \times \cdots \times X_n =:X$, where $X$ is called the
\emph{type space}. By construction, $x^{}_i$ is the $i$-th component
or coordinate of $x$, and we define $x^{}_I := (x^{}_i)^{}_{i \in I}$
as the collection of `coordinates' with indices in $I$, where $I$ is a
subset of $S$.  A \emph{population} is identified with
(or described by) a probability
vector $p=\big ( p(x) \big )_{x \in X}$ on $X$, where $p(x)$ denotes
the proportion of individuals of type $x$ in $X$. Note that we assume
the sequences to have fixed length. Additional processes that may
change this, such as copying blocks, are disregarded here; see
\cite{RB} and references therein for possible extensions.

With Fig.~\ref{fig:lifecycle} in mind, recombination may now be
modelled as follows. A new (`offspring') sequence is formed as the
`mixture' of two randomly chosen parental sequences (say $x$ and $y$)
from the population: It copies the letters of $x$ at some of its sites
and those of $y$ at all others.  If, for example, a double crossover
happens between sites $i$ and $i+1$ and between $j$ and $j+1$ ($i<j$),
then the offspring sequence reads $(x^{}_1,\ldots,x^{}_i,
y^{}_{i+1},\ldots, y^{}_j,x^{}_{j+1},\ldots, x_n)$.  The offspring
sequence replaces a randomly chosen sequence (possibly one of the
parents, but this is negligibly rare in a large population).  Viewed
differently, the offspring sequence in our example reads $x$ whenever
the parents are of the form $(x^{}_1,\ldots,x^{}_i, *,\ldots,
*,x^{}_{j+1},\ldots, x_n)$ and $(*,\ldots,*, x^{}_{i+1},\ldots,
x^{}_j,*,\ldots, *)$. Here, a `$*$' at site $i$ stands for an arbitrary
element of $X_i$, so means marginalisation. This will be helpful when
formulating the differential equation.

The sites that come from the paternal and the maternal sequences,
respectively, define a \emph{partition} ${\mathcal A}$ of $S$ into two
parts. Due to the random choice of the parents, we need not keep track
of which sequence was `maternal' and which was `paternal'. In
principle, all partitions of $S$ into two parts ($\cA=\{A_1,A_2\}$)
can be realised, via a suitable number of crossovers at suitable
positions. If no crossover happens, then the partition is $\cA
=\{S\}$, and the offspring is an exact copy of the first
parent. Reproduction with recombination according to a partition $\cA$
happens at rate $\varrho(\cA)$.

We shall introduce all notions with more care later. For now, we turn
the verbal description into a differential equation system and obtain
\begin{equation}\label{eq:ode-intro}
   \dot p^{\pa}_t (x) \, = \!  
   \sum_{{\mathcal A} \in {\mathbb P}_2 (S)}\! \varrho({\mathcal A})
   \bigl(p^{\pa}_t (x^{\pa}_{\! A_1}, *) \, 
    p^{\pa}_t (* \ts , x^{\pa}_{\! A_2}) - p^{\pa}_t (x) \bigr)
\end{equation}
for all $x \in X$, where ${\mathbb P}_2 (S)$ denotes the set of
partitions of $S$ into two parts. Eq.~\eqref{eq:ode-intro} may be
understood as a `mass balance' equation: For every ${\mathcal A} \in
{\mathbb P}_2 (S)$, sequences of type $x$ are `produced' from the
corresponding parental sequences at overall rate $\varrho({\mathcal
  A}) \, p^{}_t (x^{}_{\! A_1}, *) \, p^{}_t (*, x^{}_{\! A_2})$, where the
product reflects the random combination; at the same time, sequences
of type $x$ are lost (i.e., replaced by new ones) at overall rate
$\varrho({\mathcal A}) \, p^{}_t (x)$. Note that the case ${\mathcal
  A}=\{S\}$ provides no net contribution to $\dot p_t(x)$, since gain
and loss are equal in this case.\smallskip

The resulting ODE system appears difficult to handle, due to the large
number of possible states and the nonlinearity of the right-hand side.
In previous papers \cite{BB,MB,BH,WBB,EB-ICM,BaHu,BW}, we have
concentrated on a special case, namely, the situation in which at most
one crossover happens at any given time. That is, we restricted
attention to \emph{ordered partitions into two parts}, corresponding
to the sites before and after a single-crossover point.  We have
analysed the resulting models in continuous time (both deterministic
and stochastic), as well as in discrete time. For the deterministic
continuous-time system, a simple explicit solution is available
\cite{BB,MB}. This simplicity is due to some underlying linearity.

It is now time to tackle the case of general partitions (in continuous
time). Therefore, in this contribution, we give up the
single-crossover assumption -- and even allow for an arbitrary number
of parents in a given recombination event, which leads to partitions
with more than two parts. Even though this is not a common biological
feature, we will see that it requires little extra mathematical
effort. Also, it is a very natural structure on the lattice of
partitions of a (finite) set.
The restriction to partitions form the biologically most relevant
subset $\PP_{2} (S)$ will always be possible by a suitable choice
of the model parameters, which are the recombination rates 
$\varrho (\cA)$.

This contribution is motivated by the pioneering work of Geiringer
\cite{Gei} and Bennett \cite{Ben}, who worked on a similar system in
\emph{discrete} time (but restricted to a special type space); 
the later work of Lyubich \cite[Chapter~6]{Lyu}, who worked out
much of the underlying structure and got close to
a solution in 1992; and by
more recent work of Dawson \cite{D00,D02}, who presented a (recursive)
solution in
2000 and 2002. It relies on a certain nonlinear transformation from
(gamete or type) frequencies to suitable correlation functions, which
decouple from each other and decay geometrically. If sequences of more
than three sites are involved, this transformation must be constructed
via recursions that involve the parameters of the recombination
process.

Dawson's construction testifies to remarkable insight into the
problem. However, it is not easy to penetrate to the mathematical core
of his arguments. We therefore start at the very beginning and
formulate the model on a fairly general type space, in a
measure-theoretic framework, and allowing for arbitrary partitions.
More importantly, we put the problem into a systematic
lattice-theoretic setting; this will become the key for the
transparent construction of the solution. 
Furthermore, we establish a connection with a partitioning process
backward in time, which describes how an individual in the present
population has been pieced together from the genetic material of its 
ancestors.
This provides a link to the \emph{ancestral recombination graph} (ARG),
which is the ancestral process commonly used in  models 
of recombination in \emph{finite} populations, see 
\cite[Ch.~3.4]{Durrett}. \smallskip

The paper is organised as follows. After introducing the mathematical
objects we need and some of their properties in
Section~\ref{sec:partitions}, the general recombination equation is
discussed in Section~\ref{sec:gen-reco}. As a first step, this is done
in the setting of a measure-valued ordinary differential equation
(ODE), which is then reduced to a finite-dimensional ODE
system. Section~\ref{sec:linear} solves this system under a linearity
assumption, which is motivated by previous work, but does not give
the solution in sufficient generality. 

As a further preparation for the general solution, we study the
behaviour of the system under marginalisation in
Section~\ref{sec:marginal}.
Based on this, Section~\ref{sec:partitioning}
establishes the connection with the partitioning
process. This is followed by the derivation
of the general solution in Section~\ref{sec:gen-sol}, which is
recursive in nature and applies to the generic choice of the
recombination rates. More detailed properties of the solution are
investigated in Section~\ref{sec:properties}, while
Section~\ref{sec:degen} deals with various types of non-generic
cases. The Appendix provides some material for the treatment of
degenerate cases.

This paper builds on previous work, most importantly on
\cite{BB,MB}. Some of the results from these papers will freely be
used below, and not re-derived here (though we will always provide
precise references).

\section{Partitions, measures and recombinators}
\label{sec:partitions}

Let $S$ be a finite set, and consider the lattice $\PP = \PP (S)$ of
partitions of $S$; see \cite{Aigner} for general background on lattice
theory. When the cardinality of $S$ is $|S|=n$, the set $\PP$ contains
$B(n)$ elements, known as the \emph{Bell number}; compare \cite[A\ts
000110]{OEIS}. With $B(0):=1$, these numbers are recursively computed
as $B(n+1) = \sum_{k=0}^{n} \binom{n}{k} B(k)$ for $n\geqslant 0$,
with generating function $F(z) := \sum_{n=0}^{\infty} \frac{B(n)}{n!}
z^{n} = \exp (\ee^{z} - 1)$ and explicit formula $B(n) = \frac{1}{\ee}
\sum_{k=0}^{\infty} \frac{k^{n}}{k!}$.  A subset of relevance to us,
for the biological applications, consists of all partitions of $S$
into two parts, $\PP_{2} (S) = \bigl\{ \cA\in\PP (S) \,{\big|}\, \lvert
\cA \rvert = 2 \bigr\}$, which contains $2^{n-1} - 1$ elements.  Note
that $\PP_{2} (S)$ generates the lattice $\PP (S)$ in an obvious way.

Here, we write a partition of $S$ as $\cA = \{ A_{1}, \dots , A_{r}
\}$, where $r = |\cA|$ is the number of its parts (or blocks), and one has $A_{i}
\cap A_{j} = \varnothing$ for all $i\ne j$ together with $A_{1} \cup
\dots \cup A_{r} = S$. The natural ordering relation is denoted by
$\preccurlyeq$, where $\cA \preccurlyeq \cB$ means that $\cA$ is
\emph{finer} than $\cB$, or that $\cB$ is \emph{coarser} than $\cA$.
The conditions $\cA \preccurlyeq \cB$ and $\cB \succcurlyeq \cA$ are
synonymous, while $\cA \prec\cB$ means $\cA \preccurlyeq \cB$ together
with $\cA \ne \cB$.

The joint refinement of two partitions $\cA$ and $\cB$ is written as
$\cA \wedge \cB$, and is the coarsest partition below $\cA$ and $\cB$.
The unique \emph{minimal} partition within the lattice $\PP$ is
denoted as $\pmin = \{ \{x\} \mid x \in S \}$, while the unique
\emph{maximal} one is $\pmax = \{ S \}$.  When $\cA \preccurlyeq \cB$,
we also employ the interval notation $[\cA , \cB] := \{ \cC \mid \cA
\preccurlyeq \cC \preccurlyeq \cB \}$. For a general subset $\GG$ of
$\PP$, we write the complement as $\GG_{}^{\mathsf{c}} = \PP\setminus
\GG$.  Finally, when $\GG \subseteq \PP$, the coarsest partition below
all elements of $\GG$ is denoted by $\bigwedge \! \GG$. Note that
$\bigwedge\! \varnothing = \pmax$ by convention.

When $U$ and $V$ are disjoint (finite) sets, two partitions
$\cA\in\ts\PP(U)$ and $\cB\in\ts\PP(V)$ can be joined to form an
element of $\PP(U\nts \cup V)$. We denote such a \emph{joining} by
$\cA\sqcup \cB$, and similarly for multiple joinings.  Conversely, if
$U\nts\subset S$, a partition $\cA\in\ts\PP(S)$, with $\cA = \{
A_{1}, \dots , A_{r} \}$ say, defines a unique partition of $U$ by
restriction. The latter is denoted by $\cA|^{\pa}_{U}$, and its parts
are precisely all non-empty sets of the form $A_{i} \cap U$ with
$1\leqslant i \leqslant r$. \medskip

Let now $S=\{ 1,2,\dots ,n\}$ and define $X = X_{1} \times \dots
\times X_{n}$, where each $X_{i}$ is a locally compact space. In many
concrete applications, the $X_{i}$ will be finite sets, but we do not
make such a restriction as it is neither necessary nor
desirable.  In particular, there are situations in quantitative
genetics \cite{Buerger} that will profit from the more general setting
we employ here.

When $S$ and $X$ are given, we denote the natural projection to the
$i$th component by $\pi_{i}$, so $\pi_{i} (X) = X_{i}$. Similarly, for
an arbitrary non-empty subset $U\nts\subseteq S$, we use the notation
$\pi^{\pa}_{U} \!  : \; X \longrightarrow X^{\pa}_{U}:= \bigtim_{i\in
  U} X_{i}$ for the projection to the subspace $X^{\pa}_{U}$.

Let $\cM (X)$ denote the space of finite, regular Borel measures on
$X$, equipped with the usual total variation norm $\| . \|$, which
makes it into a Banach space. Also, we need the closed subset (or
cone) $\cM_{+} (X)$ of positive measures, which we mean to include the
zero measure. Within $\cM_{+} (X)$, we denote the closed subset of
probability measures by $\cP (X)$. Note that $\cM_{+} (X)$ and
$\cP(X)$ are convex sets.  The restriction of a measure $\mu\in\cM
(X)$ to a subspace $X^{\pa}_{U}$ is written as $\pi^{\pa}_{U} . \mu :=
\mu \circ \pi^{-1}_{U}$, which is consistent with marginalisation of
measures. For any Borel set $A\subseteq X^{\pa}_{U}$, one thus has the
relation $\bigl(\pi^{\pa}_{U} . \mu\bigr) (A) = \mu \bigl(
\pi^{-1}_{U} (A)\bigr)$.

Given a measure $\nu \in \cM(X)$ and a partition $\cA = \{ A_{1},
\dots ,A_{r} \} \in \PP$, we define the mapping $R^{\pa}_{\nts \cA} \! :
\, \cM(X)\longrightarrow \cM(X)$ by $\nu \mapsto R^{\pa}_{\nts \cA}
(\nu)$ with $R^{\pa}_{\nts \cA} (0) := 0$ and, for $\nu\ne 0$,
\begin{equation}\label{eq:def-recomb}
  R^{\pa}_{\nts \cA} (\nu) \, := \, \frac{1}{\| \nu \|^{r-1}}
  \bigotimes_{i=1}^{r} \bigl(\pi^{\pa}_{A_{i}} . \nu\bigr)
\end{equation}
where the product is (implicitly) `site ordered', i.e.\ it matches the
ordering of the sites as specified by the set $S$. We shall also use
site ordering for cylinder or product sets. We call a mapping of type
$R^{\pa}_{\!\cA}$ a \emph{recombinator}. Note that recombinators are
nonlinear whenever $\cA \ne \pmax$.

\begin{prop}\label{prop:gen-props}
  Let\/ $S=\{ 1,2,\dots ,n \}$ and\/ $X=X_{1} \times \dots \times
  X_{n}$ as above. Now, let\/ $\cA \in
  \PP(S)$ be arbitrary, and consider the corresponding recombinator as
  defined by Eq.~\eqref{eq:def-recomb}. Then, the following assertions
  are true.
\begin{enumerate}\itemsep=1pt
\item $R^{\pa}_{\! \cA}$ is positive homogeneous of degree\/ $1$, which
  means that\/ $R^{\pa}_{\! \cA} (a \nu) = a R^{\pa}_{\! \cA} (\nu)$ holds
  for all\/ $\nu\in\cM(X)$ and\/ $a\geqslant 0$.
\item $R^{\pa}_{\! \cA}\vph$ is globally Lipschitz on\/ $\cM(X)$, with
  Lipschitz constant\/ $L\leqslant 2 \lvert \cA \rvert + 1$.
\item $\| R^{\pa}_{\! \cA} (\nu)\| \leqslant \| \nu \|\vph$ holds for
  all\/ $\nu\in\cM(X)$.
\item $R^{\pa}_{\! \cA}\vph$ maps\/ $\cM_{+} (X)$ into itself.
\item $R^{\pa}_{\! \cA}\vph$ preserves the norm of positive measures,
  and hence also maps\/ $\cP (X)$ into itself.
\end{enumerate}
\end{prop}

\begin{proof}
  Claims (1) and (4) are elementary consequences of the definition.
  Claim (3) follows easily from the arguments used in
  \cite[Sec.~3.1]{BB}, while (5) is clear from the observation that
  $\| \nu \| = \nu (X)$ for $\nu \ge 0$, which then implies $\|
  R^{\pa}_{\! \cA} (\nu)\| = \| \nu \|$ by standard arguments; compare
  \cite[Fact~2]{BB}.

  It remains to prove (2).  Let $\cA = \{ A_{1}, \dots , A_{r}\}$ be a
  partition with $r$ parts, where $1 \leqslant r \leqslant n$. Then,
  one can show inductively in $r$ that
\begin{equation}\label{eq:Lip-estimate}
  \Bigl\| \bigotimes_{i=1}^{r} (\pi^{\pa}_{\! A_{i}}\nts .\mu) -
  \bigotimes_{j=1}^{r} (\pi^{\pa}_{\! A_{j}}\nts .\nu) 
      \Bigr\| \, \leqslant \,
  \| \mu - \nu \|\, \sum_{k=0}^{r-1} \| \mu \|^{k} \| \nu \|^{r-1-k},
\end{equation}
for arbitrary $\mu,\nu \in \cM(X)$. If one of them is the zero
measure, the Lipschitz estimate is a consequence of claim (3).

Let now $\mu,\nu\in\cM(X)$ both be non-zero, and assume that
$\| \mu \| \leqslant \| \nu \|$. Using Eq.~\eqref{eq:Lip-estimate} 
together with positive homogeneity of $R^{\pa}_{\nts \cA}$, we have
\begin{align} 
  \nonumber
  \bigl\| R^{\pa}_{\!\cA} (\mu) - R^{\pa}_{\!\cA} (\nu) \bigr\| \,& = \,
  \Bigl\| \| \mu \| \Bigl(\bigotimes_{i=1}^{r} \pi^{\pa}_{\! A_{i}} . 
   \frac{\mu}{\|\mu\|} - \bigotimes_{j=1}^{r} \pi^{\pa}_{\! A_{j}} . 
   \frac{\nu}{\|\nu\|}\Bigr)
  + \bigl(\|\mu\| - \|\nu\| \bigr) \bigotimes_{k=1}^{r}
   \pi^{\pa}_{\! A_{k}} . \frac{\nu}{\|\nu\|} \Bigr\| \\[1mm]
   & \leqslant \, 
    r\, \min \bigl( \|\mu\| ,\| \nu \| \bigr) \ts \Bigl\| 
    \frac{\mu}{\|\mu\|} - \frac{\nu}{\|\nu\|}\Bigr\| 
    \, + \, \| \mu - \nu \|  \ts ,  \label{align:min}
\end{align}
where we used that $\mu/\|\mu\|$ and $\nu/\|\nu\|$ are measures of
norm $1$.  Next, one has
\[
   \Bigl\| \frac{\mu}{\|\mu\|} - 
    \frac{\nu}{\|\nu\|}\Bigr\| \, \leqslant \,
    \frac{2\, \| \mu - \nu \|}{\max \bigl( \|\mu\| , \| \nu \| \bigr)} 
\]
which, on inserting into Eq.~\eqref{align:min}, gives the inequality
\[
   \bigl\| R^{\pa}_{\!\cA} (\mu) - R^{\pa}_{\!\cA} (\nu) \bigr\|
   \, \leqslant \, (2r+1)\, \| \mu - \nu \|
\]
from which our claim follows, with $L \leqslant 2 r + 1$.
\end{proof}

\begin{remark}
  Let us mention that there is an alternative way to see the Lipschitz
  property, at least for finite $X$.  This is because the dynamics may
  then be reformulated in terms of a chemical reaction system. This is
  a large class of models, for which a substantial body of theory is
  available; see \cite{Son} for a review. In particular, the Lipschitz
  property applies under very general conditions, which are satisfied
  in our case.  \exend
\end{remark}

The estimate used above is rather coarse, but suffices for our needs
here. When using the invariance of $\cP (X)$, one simply employs
Eq.~\eqref{eq:Lip-estimate} from the last proof to establish the
following consequence.

\begin{coro}
  On\/ $\cP (X)$, with\/ $\cA\in\PP(S)$, the 
  recombinator\/ $R^{\pa}_{\! \cA}$ is 
  Lipschitz with $L\leqslant \lvert \cA \rvert$.  \qed
\end{coro}

Before we embark on the recombination equation and its solution, we
need to establish one technical result on the relation between
recombinators and projectors to subsystems, as defined by non-empty
sets $U\nts\subset S$. Here, the system on which
a recombinator acts is marked by an upper index $S$ (for the full
system) or $U$ (for the subsystem). Later, we will drop this index 
whenever the meaning is unambiguous.

\begin{lemma}\label{lem:technical}
  Let\/ $S$ be a finite set as above, and\/ $\cA = \{ A_{1}, \dots ,
  A_{r} \} \in \PP(S)$ an arbitrary partition. If\/ $U\nts\subseteq S$
  is non-empty and $\omega\in\cM_{+} (X)$, one has
\[
     \pi^{\pa}_{U} . \bigl( R^{\ts S}_{\nts \cA}  (\omega) \bigr) \, = \, 
     R^{\ts U}_{\nts \cA|^{\pa}_{U}} (\pi^{\pa}_{U}  . \ts \omega) \ts ,
\]   
where\/ $\cA|^{\pa}_{U} \in \PP(U)$ and the upper index of a
recombinator indicates on which measure space it acts.
\end{lemma}

\begin{proof}
  We shall show the claimed identity by verifying it for certain
  `rectangular' measurable sets that suffice for the equality of the
  two measures as Baire measures, and then rely on the unique
  extension of Baire measures to regular Borel measures; compare the
  discussion around \cite[Fact~1]{BB}. To this end, let
  $\cA\in\ts\PP(S)$ be given as $\cA = \{ A_{1}, \dots , A_{r} \}$,
  and write $\cA|^{\pa}_{U} \in \PP(U)$ as $\cA|^{\pa}_{U} = \{ V_{1},
  \dots , V_{s} \}$, where $s\leqslant r$. Each $V_i$ is contained in
  precisely one part of $\cA$.  Without loss of generality, we may
  thus assume that $V_{i} = A_{i} \cap U$ for $1\leqslant i \leqslant
  s$, while $A_{j} \cap U = \varnothing$ for all $j > s$ (when $r=s$,
  this case does not occur).

  Now, let $E_{i}$ be a Borel set in $X^{\pa}_{\nts V_{i}} =
  \pi^{\pa}_{V_i} (X)$, and consider the corresponding `rectangular'
  set $E = E_{1} \times \dots \times E_{s} \subseteq
  X_{U}$, where we again assume that the product observes proper site
  ordering.  When evaluated on $E$, the right-hand side of our claim
  gives
\[
\begin{split}
   \bigl(R^{\ts U}_{\cA|^{\pa}_{U}} (\pi^{\pa}_{U} . \ts  \omega) \bigr)
   ( E ) \, & = \, \frac{1}{\| \pi^{\pa}_{U}  .\ts \omega \|^{s-1}} 
   \bigotimes_{i=1}^{s} \bigl(\pi^{(U)}_{V_i} \! . \ts 
    (\pi^{\pa}_{U}  . \ts \omega) \bigr)
   ( E_{1} \times \dots \times E_{s} ) \\
   & = \, \frac{1}{\| \omega \|^{s-1}} 
   \bigotimes_{i=1}^{s} (\pi^{\pa}_{V_i} \nts . \ts \omega) 
   ( E_{1} \times \dots \times E_{s} ) \ts ,
\end{split}   
\]
where we have used that $\| \pi^{\pa}_{\nts U} . \ts \omega \| = \|
\omega \|$ and where $\pi^{(U)}$ denotes a projector that is only
defined on the subspace $X^{\pa}_{U}$.  This is now to be compared
with the evaluation of the left hand side of the claim, which gives
(with $U^{\mathsf{c}} = S \setminus U$)
\[
\begin{split}
     \bigl( \pi^{\pa}_{\nts U} .\ts R^{S}_{\nts\cA} 
     (\omega)\bigr) (E) \, & = \, 
     R^{S}_{\nts\cA} (\omega) \bigl( \pi^{-1}_{U} (E) \bigr) 
     \, = \,  \frac{1}{\| \omega \|^{r-1}}
        \bigotimes_{i=1}^{r} \, 
       ( \pi^{\pa}_{\! A_i}\nts . \ts \omega ) 
     ( E \times X^{\pa}_{U_{\phantom{I}}^{\mathsf{c}}} ) \\[1mm]
     & = \, \frac{1}{\| \omega \|^{r-1}} \,
      \Bigl( \bigotimes_{i=1}^{s}
      (\pi^{\pa}_{\! A_i} \nts . \ts \omega )\nts \Bigr) 
       \otimes \Bigl( \bigotimes_{j=s+1}^{r}
      (\pi^{\pa}_{\! A_j} \nts . \ts \omega)\nts \Bigr) 
         ( E_{1} \times \dots \times E_{s} \times 
         X^{\pa}_{U_{\phantom{I}}^{\mathsf{c}}} ) \\[1mm]
     & = \, \frac{1}{\| \omega \|^{r-1}}  \biggl( \, \prod_{i=1}^{s}
         \omega \bigl( \pi^{-1}_{\! A_i} 
        ( E_{i} \times \nts 
        X^{\pa}_{\! A_{i} \setminus V_{i}})
         \bigr)\! \biggr) \biggl( \, \prod_{j=s+1}^{r} \omega
         \bigl( \pi^{-1}_{\! A_{j}} (X^{\pa}_{\! A_j})\bigr)\! 
         \biggr) \\[1mm]
     & =    \, \frac{1}{\| \omega \|^{r-1}}  \biggl( \, \prod_{i=1}^{s}
         \omega \bigl( \pi^{-1}_{V_i} (E_{i}) \bigr)\! \biggr) 
         \biggl( \, \prod_{j=s+1}^{r} \omega(X)\! \biggr) \\[1mm]
      & = \,    \frac{\|\omega \|^{r-s}}{\| \omega \|^{r-1}} \;
         \prod_{i=1}^{s} \bigl( \pi^{\pa}_{V_i} 
         \nts . \ts \omega \bigr) (E_{i})
         \, = \, \frac{1}{\| \omega \|^{s-1}}  \bigotimes_{i=1}^{s}
          \bigl( \pi^{\pa}_{V_{i}} \nts . \ts \omega\bigr)
            ( E_{1} \times \dots \times E_{s}) \ts , 
\end{split}
\]
which agrees with our previous expression and thus proves the claim.
\end{proof}

Lemma ~\ref{lem:technical} has some interesting consequences for the
structure of recombinators. We only state the result and omit the
proof, as it is analogous to the corresponding result in \cite{BB}.

\begin{prop}
    Let\/ $\cA,\cB\in\PP(S)$. On\/ $\cM_{+} (X)$, the corresponding
    recombinators satisfy
\[
    R^{\pa}_{\nts \cA} R^{\pa}_{\cB}  \, = \, R^{\pa}_{\nts \cA \wedge \cB} \ts .
\]
  In particular, each recombinator is an idempotent and any two
  recombinators commute.    \qed
\end{prop}

Below, we shall also need probability vectors on $\PP = \PP(S)$, which
form the compact space $\cP(\PP)$. An interesting subclass of them can
be constructed as follows.  Let $f \! : \, \PP \longrightarrow [0,1]$
be a function on the lattice $\PP$, and consider
\begin{equation}\label{eq:c-def}
    c (\cA) \, :=  \sum_{\substack{\GG\subseteq \PP 
   \\ \bigwedge\! \GG = \cA}}\, \prod_{\cB\in\GG}
      \bigl(1-f(\cB\ts )\bigr)
      \prod_{\cC\in \GGc} \! f(\cC)
\end{equation}
for $\cA\in\ts\PP$, which clearly satisfies $c (\cA)\geqslant 0$.
Moreover, one has
\begin{equation}\label{eq:c-normed}
   \sum_{\cA\in\ts\PP} c (\cA) \, =  \sum_{\GG\subseteq\PP}\,
   \prod_{\cB\in\GG} \bigl(1-f(\cB\ts )\bigr) \prod_{\cC\in \GGc} f(\cC)
   \, =  \prod_{\cA\in\ts\PP} \bigl( (1-f(\cA) ) + f(\cA)\bigr)
   \, = \, 1 \ts ,
\end{equation}
wherefore $c\in\cP(\PP(S))$, which means that $c$ is a probability
vector on the lattice $\PP(S)$.

At this point, we have gathered the core material to embark on the
discussion of the recombination equation and its properties.

\section{The general recombination equation}
\label{sec:gen-reco}

With the tools at hand, it is now rather obvious how to generalise the
`mass balance' equation \eqref{eq:ode-intro} of the Introduction to
our measure-theoretic setting. Within the Banach space $(\cM(X),
\|.\|)$, we thus consider the nonlinear ODE
\begin{equation}\label{eq:reco-eq}
   \dot{\omega}^{\pa}_{t} \, = \!
   \sum_{\cA\in \PP (S)} \! \! \varrho(\cA) \ts
   \bigl( R^{\pa}_{\nts\cA} - \pmax \bigr) (\omega^{\pa}_{t})
\end{equation}
with non-negative numbers $\varrho(\cA)$ that have the meaning of
\emph{recombination rates} in our context. They are written in this
way because we consider $\varrho \! : \, \PP (S) \longrightarrow \RR$
as an element of the \emph{M\"{o}bius algebra} over $\RR$; see
\cite{Aigner,SD} for background.  We will usually assume that an
initial condition $\omega^{\pa}_{0} \in \cM(X)$ for $t=0$ is given for
the ODE \eqref{eq:reco-eq}, and then speak of the corresponding
\emph{Cauchy problem} (or initial value problem).  Biologically,
Eq.~\eqref{eq:reco-eq} describes the change in composition of a
population in which offspring is produced by piecing together
sequences from various parents according to the collection of the
$\varrho(\cA)$ with $\cA \in \PP (S)$. More precisely, with rate
$\varrho(\cA)$ for every $\cA = \{A_1, \ldots, A_r\}$, the sites in
$A_1$ inherited from parent $1$ are reassociated with the sites in
$A_2$ inherited from parent $2$, and with the sites in $A_i$ inherited
from parent $i$ for $2<i\le r$.

\begin{remark}\label{rem:balance}
  Since $R^{\pa}_{\pmax} = \pmax$, the value of $\varrho (\pmax)$ is
  immaterial, and the corresponding term could clearly be omitted on
  the right-hand side of Eq.~\eqref{eq:reco-eq}. We nevertheless keep
  it here, as it will become useful in connection with the reduction
  to subsystems.
   
   Another way to write the ODE is
\[
    \dot{\omega}^{\pa}_{t} \, = \, -\rtot{} \, \omega^{\pa}_{t}
    \, + \! \! \sum_{\cA\in\PP (S)} \! \! \varrho(\cA) \, R^{\pa}_{\nts \cA}
    (\omega^{\pa}_{t} ) \, = \, - \bigl(\rtot{} - \varrho(\pmax)\bigr)
    \, \omega^{\pa}_{t}  \, +  \sum_{\cA \prec \pmax} 
    \varrho(\cA) \, R^{\pa}_{\nts \cA} (\omega^{\pa}_{t} )
\]   
   with $\rtot{} := \sum_{\cA\in\PP (S)} \varrho (\cA)$ being the total
   recombination rate. 
\exend
\end{remark}

With $\varPhi := \sum_{\cA\in \PP (S)} \varrho(\cA) \ts
   \bigl( R^{\pa}_{\nts\cA} - \pmax \bigr)$, we can now simply write
\begin{equation}\label{eq:short}
     \dot{\omega}^{\pa}_{t} \, = \, \varPhi (\omega^{\pa}_{t}) \ts ,
\end{equation}
but we must keep in mind that $\varPhi$ is a nonlinear operator.
Nevertheless, one has the following basic result; compare \cite{Amann}
for background on ODEs on Banach spaces.

\begin{prop}\label{prop:gensol}
  Let\/ $S$ be a finite set and\/ $X$ the corresponding locally
  compact product space as introduced above. Then, the Cauchy problem
  of Eq.~\eqref{eq:reco-eq} with initial condition\/ $\omega^{\pa}_{0}
  \in \cM (X)$ has a unique solution.  Moreover, the cone\/ $\cM_{+}
  (X)$ is forward invariant, and the flow is norm-preserving on\/
  $\cM_{+} (X)$.  In particular, $ \cP (X)$ is forward invariant under
  the flow.
\end{prop}

\begin{proof}
  By part (2) of Proposition~\ref{prop:gen-props}, all $R^{\pa}_{\nts
    \cA}$ are globally Lipschitz on $\cM(X)$, which is then also true
  of $\varPhi$. Therefore, the uniqueness statement for the Cauchy
  problem is clear.

  Next, let $\nu\in\cM_{+} (X)$ and consider an arbitrary Borel set
  $E\subset X$ with $\nu (E)=0$. Then, with $\rtot{} =
  \sum_{\cA\in\ts\PP(S)}\varrho(\cA)$ as above in
  Remark~\ref{rem:balance}, we have
\[
     \bigl( \varPhi (\nu) \bigr) (E) \, = \, - \rtot{} \, 
      \nu(E) \, + \! \! \sum_{\cA\in\ts\PP(S)} \!\! \varrho(\cA) \ts 
      \bigl( R^{\pa}_{\nts \cA} (\nu)\bigr)
      (E) \, \geqslant \, 0
\]
because $\nu (E)=0$ by assumption and all other terms are non-negative
as a result of part (4) of Proposition~\ref{prop:gen-props} together
with $\varrho (\cA)\geqslant 0$ for all $\cA\in\ts\PP(S)$.  Positive
invariance of the closed cone $\cM_{+} (X)$ now follows from a classic
continuity argument; see \cite[Thm.~16.5 and Rem.~16.6]{Amann} for
details and \cite[Thm.~1]{BB} for the analogous argument in the
single-crossover model.

When $\nu$ is a positive measure, one has $\nu (X) = \| \nu \| = \|
R^{\pa}_{\nts \cA} (\nu) \| = \bigl( R^{\pa}_{\nts \cA} (\nu)\bigr)
(X)$ as a consequence of Proposition~\ref{prop:gen-props}, and hence
\[
    \bigl( \varPhi (\nu)\bigr) (X) \, = \, - \rtot{} \, \nu(X)
     \,  +  \! \! \sum_{\cA\in\ts\PP(S)}\! \! \varrho(\cA)\, 
     \bigl( R^{\pa}_{\nts \cA} (\nu) \bigr) (X) \, = \, 0 \ts ,
\]    
which implies the preservation of the norm of a positive measure under
the forward flow. The last claim is then obvious.
\end{proof}

In view of our underlying biological problem, we restrict our
attention to the investigation of the recombination equation on the
cone $\cM_{+} (X)$, and on $\cP (X)$ in particular. To proceed, we
will first bring the ODE \eqref{eq:reco-eq} into a simpler form.  To
this end, observe the structure of $\varPhi$, which suggests that, as
the flow proceeds forward in time from an initial measure $\nu\in\cP
(X)$, the solution picks up fractions of components of the form
$R^{\pa}_{\cB} (\nu)$ for various or even all $\cB\in\ts\PP(S)$,
depending on which recombination rates are positive and which
partitions can thus be reached in the course of time.

Let $|S|=n$ and fix some $\nu\in\cM_{+}(X)$.  Now, define the 
(finite-dimensional) set
\[
   \varDelta_{\ts \nu} \, := \, 
    \bigl\{\ts \textstyle{\sum_{\cC\in\ts\PP(S)}}
   \, q(\cC)\ts R^{\pa}_{\cC} (\nu) \mid q \in \cP (\PP (S)) \bigr\} ,
\]
which is the closed convex set that consists of all convex linear
combinations of the measures $R^{\pa}_{\cC} (\nu)$ with $\cC\in\PP
(S)$.  In fact, when the rates $\varrho (\cA)$ run through all
non-negative values, $\varDelta_{\ts \nu}$ is the smallest closed
convex set that contains all measures which can be reached in the
course of recombination from the initial measure $\nu$. The dimension
of $\varDelta_{\ts \nu}$ depends on the nature of $\nu$, and can even
be $0$. When the dimension is maximal (meaning $B(n) - 1$, which is
the generic case), $\varDelta_{\ts \nu}$ is a simplex and the $B(n)$
measures $R^{\pa}_{\cC} (\nu)$ are the \emph{extremal} measures of the
simplex in the sense of convex analysis. Observe that
\[
    R^{\pa}_{\cB} (\varDelta_{\ts \nu}) \, \subseteq \, \varDelta_{\ts \nu}
\]
holds for any $\cB \in \PP(S)$, which follows from a straightforward
though slightly technical calculation (details of which are given
below). This now suggests the ansatz
\begin{equation}\label{eq:sol-ansatz}
   \omega^{\pa}_{t} \, =  \! \sum_{\cC\in\ts\PP(S)}\! \! a^{\pa}_{t} 
   \nts\nts (\cC) \, R^{\pa}_{\cC} (\omega^{\pa}_{0})
\end{equation}
for the solution of our above Cauchy problem, with (generic) initial
condition $\omega^{\pa}_{0}$ and coefficient functions $a^{\pa}_{t}$,
the latter thus with $a^{\pa}_{0} (\cC) = \delta^{\pa}_{\cC,\pmax}$.
For each $t$, we view $a^{\pa}_{t} \! : \, \PP (S) \longrightarrow
\RR$ as an element of the M\"{o}bius algebra.  We shall prove a
posteriori that the strategy of Eq.~\eqref{eq:sol-ansatz} works, and
that the original Banach space ODE is thus reduced to a
finite-dimensional system of ODEs, whose solution consists of a
one-parameter family of probability vectors on $\PP(S)$. The case of
non-generic initial conditions will be discussed afterwards.\medskip

To proceed, we now have to calculate the action of $\varPhi$ on a
general measure $\omega \in \varDelta_{\ts\nu}$. It is convenient to
first consider probability measures, as the extension to general
positive measures is then immediate via the positive homogeneity of
the recombinators. So, let $\nu\in\cP(X)$, which means
$\varDelta_{\ts\nu} \subseteq \cP(X)$. Consider a single partition $\cB
= \{B_{1}, \dots , B_{r} \}$ and observe that
\[
\begin{split}
  R^{\pa}_{\cB} (\omega) \, & = \,
  \bigotimes_{i=1}^{r}\ts \pi^{\pa}_{\! B_i}\nts . \ts \omega  \, = \,
  \bigotimes_{i=1}^{r} \ts \pi^{\pa}_{\! B_i} .
    \biggl(\, \sum_{\cC\in\ts\PP(S)} \! q(\cC) \ts 
    R^{\pa}_{\cC} (\nu)\! \biggr)  \\[1mm]
  & = \,  \bigotimes_{i=1}^{r}\, \sum_{\cC\in\ts\PP(S)} \! q(\cC) \,
   \bigl( \pi^{\pa}_{\! B_i} . R^{\pa}_{\cC} (\nu) \bigr) \, =
   \sum_{\substack{\cC_{1}\nts , \dots , \cC_{r} \\
   \in\ts \PP(S)}} \biggl(\, \prod_{j=1}^{r}\, q(\cC_{j})\! \biggr)\,
   \bigotimes_{i=1}^{r} R^{\ts B_i}_{\nts \cC_{i}|^{\pa}_{\nts B_i}} \! 
   (\pi^{\pa}_{\! B_i} .\ts  \nu) \ts ,
\end{split}
\]
where we have used Lemma~\ref{lem:technical} in the last step.
Observing that the product measure in the last expression is a measure
of the form $R^{\pa}_{\nts\cA} (\nu)$ for some $\cA\in\ts\PP(S)$ with
$\cA\preccurlyeq \cB$, and that each such $\cA$ must occur here, we
see that
\begin{equation}\label{eq:product-structure}
    R^{\pa}_{\cB} (\omega) \, = \,  R^{\pa}_{\cB} \biggl(\,
    \sum_{\cC\in\ts\PP(S)} \! q(\cC) \, R^{\pa}_{\cC} (\nu)\! \biggr)
    \, = \sum_{\udo{\cA\ts}\nts 
    \preccurlyeq \cB} \biggl( \, \prod_{i=1}^{|\cB|} 
    \sum_{\substack{\cC\in\ts\PP(S) \\ \cC|^{\pa}_{\nts B_i} 
    \nts = \ts \cA|^{\pa}_{\nts B_i}}} 
    \! \! \! q(\cC)\! \biggr)\ts  R^{\pa}_{\nts \cA} (\nu) \ts ,
\end{equation}
where we use a dot under a symbol to mark it as the summation
variable, thus following the notation of \cite{Aigner}.  Next, if
$\alpha \geqslant 0$, Proposition~\ref{prop:gen-props} implies that
\[
     \alpha \ts \omega \, = \! 
     \sum_{\cC\in\ts\PP(S)} \!\! \alpha \, q(\cC) \, 
     R^{\pa}_{\cC} (\nu) \, = \sum_{\cC\in\ts\PP(S)} \! q(\cC) \,
     R^{\pa}_{\cC} (\alpha \ts \nu)
\]
are two equivalent ways to write $\omega' = \alpha\ts \omega$.
Consequently, we extend our previous formula as
\begin{equation}\label{eq:extend-gamma}
   R^{\pa}_{\cB} (\omega) \, =  \sum_{\udo{\cA\ts}\nts \preccurlyeq\cB}
   \gamma(q;\cA,\cB\ts ) \, R^{\pa}_{\nts \cA} (\nu) \ts ,
\end{equation}
where we define the function $\gamma$  as
\begin{equation}\label{eq:def-gamma}
     \gamma(q; \cA,\cB\ts ) \, := \, \frac{1}{\| q \|_{1}^{|\cB|-1}} 
     \, \prod_{i=1}^{|\cB|} \sum_{\substack{\cC\in\ts\PP(S) \\ 
      \cC|^{\pa}_{\nts B_i} \nts = \cA|^{\pa}_{\nts B_i}}} \! \! \! q (\cC)
\end{equation}
for any $\cA,\cB\in \PP(S)$ with $\cA\preccurlyeq\cB$, and as
$\gamma(q;\cA,\cB\ts )=0$ otherwise. In Eq.~\eqref{eq:def-gamma}, $q$
may be any real vector of dimension $B(n)$, even though we have only
considered positive ones above. Also, one has $\gamma (0; \cA,\cB) =
0$ for consistency.

\begin{remark}\label{rem:prob-vec}
  When $q$ is a probability vector on $\PP (S)$, one has $\| q
  \|^{\pa}_{1} = q (\PP (S)) = 1$. The right-hand side of
  Eq.~\eqref{eq:extend-gamma} can then be read as a \emph{product of
    marginalised probabilities} for the subsystems defined by the
  parts $B_i$ of the partition $\cB$. This structure will be made
  explicit in Section~\ref{sec:marginal} (see
  Lemma~\ref{lem:technical} in particular) and will later pave the way
  to a recursive solution of the recombination equation.  \exend
\end{remark}

For fixed $q$, the function $\gamma(q;\cdot,\cdot)$ is an element of
the \emph{incidence algebra} of our lattice $\PP(S)$ over the field
$\RR$.  Following \cite{Aigner}, we denote this algebra by $\AAA
(\PP(S))$, with convolution as multiplication. The latter is defined
as
\[
    \bigl(\alpha * \beta \bigr) (\cA,\cB\ts ) \, = \!
    \sum_{\cC\in [\cA,\cB ]} \! \alpha (\cA,\cC) 
    \, \beta (\cC,\cB\ts ) \ts .
\]
Note that this definition automatically gives $ (\alpha * \beta)
(\cA,\cB\ts ) =0$ if $\cA\not\preccurlyeq\cB$, as it must.  Important
elements of $\AAA (\PP(S))$ include the \emph{unit} $\delta$, defined
by $\delta (\cA,\cB\ts ) = \delta^{}_{\! \cA,\cB}$, the \emph{zeta
  function} $\zeta$, defined by $\zeta(\cA,\cB\ts ) = 1$ for
$\cA\preccurlyeq\cB$ and $\zeta(\cA,\cB\ts ) = 0$ otherwise, and the
\emph{M\"{o}bius function} $\mu$, which is the multiplicative (left
and right) inverse of $\zeta$. One then has $\zeta\nts\nts * \mu = \mu
* \zeta = \delta$.  We refer to \cite{SD} for details and more
advanced aspects of incidence algebras.

With $\rtot{} = \sum_{\cA\in\ts\PP(S)} \varrho (\cA)$ and $0 \leqslant
\omega \in \varDelta_{\ts \nu}$, Eq.~\eqref{eq:extend-gamma} allows us
to continue as follows,
\begin{align}
  \nonumber
  \varPhi (\omega) \, & = \,
  \varPhi \biggl(\, \sum_{\cC\in\ts\PP(S)} \! q(\cC) \ts 
    R^{\pa}_{\cC} (\nu)\! \biggr) \,  = 
   \sum_{\cB\in\ts\PP(S)} \! \varrho(\cB\ts )\ts 
   \bigl( R^{\pa}_{\cB} - \pmax \bigr)
   \sum_{\cC\in\ts\PP(S)} \! q(\cC) \ts 
   R^{\pa}_{\cC} (\nu) \\[2mm]
  & = - \rtot{}\, \omega \; + \sum_{\cB\in\ts\PP(S)} 
    \! \varrho(\cB\ts )
   \sum_{\udo{\cA\ts}\nts \preccurlyeq \cB} \gamma(q;\cA,\cB\ts ) \,  
    R^{\pa}_{\cA} (\nu) 
    \label{eq:two-star} \\[1mm]
  & = \sum_{\cA\in\ts\PP(S)} \! \! R^{\pa}_{\nts\cA} (\nu)
   \Bigl( - \rtot{}\, q(\cA)   \, +  
   \sum_{\udo{\cB}\succcurlyeq\cA} \gamma(q;\cA,\cB\ts )\, 
     \varrho (\cB\ts ) \Bigr) \ts ,
     \nonumber
\end{align}
which is the desired action of $\varPhi$ on elements of $\cM_{+}(X)$.

Inserting Eq.~\eqref{eq:sol-ansatz} into the recombination equation
\eqref{eq:reco-eq} now gives the following result.

\begin{lemma}\label{lem:reduction}
  The ansatz \eqref{eq:sol-ansatz} for the solution\/
  $\omega^{\pa}_{t}$ of the recombination equation \eqref{eq:reco-eq}
  on\/ $\cM_{+} (X)$ leads to a system of induced ODEs for the
  coefficient functions\/ $a^{\pa}_{t}\nts$, namely
\begin{equation}\label{eq:gen-reco-coeff-DGL}
   \dot{a}^{\pa}_{t} (\cA) \, = \, - \rtot{}\,
   a^{\pa}_{t} (\cA)\, + \sum_{\udo{\cB} \succcurlyeq \cA}
  \gamma ( a^{\pa}_{t} \nts ; \cA,\cB\ts ) \ts \varrho (\cB\ts )
\end{equation}
for $\cA\in\ts\PP(S)$, where $\gamma$ is defined as
in Eq.~\eqref{eq:def-gamma}.    
\end{lemma}

\begin{proof}
  Clearly, Eq.~\eqref{eq:gen-reco-coeff-DGL} is the result of a
  comparison of coefficients, based on the original equation
  \eqref{eq:reco-eq}, its short form \eqref{eq:short}, and
  Eq.~\eqref{eq:two-star}. For a generic $\nu\in\cM_{+} (X)$, this is
  justified by the fact that the measures $R^{\pa}_{\nts \cA} (\nu)$
  with $\cA\in\PP(S)$ then are the extremal measures of the
  forward-invariant simplex $\varDelta_{\ts\nu}$. Since the set of
  generic $\nu$ is dense in $\cM_{+}(X)$, and the solution of the
  Cauchy problem depends continuously on the initial condition, the
  extension to all $\nu\in\cM_{+} (X)$ is consistent.
\end{proof}

Note that our above calculation was based upon the action of
recombinators on \emph{positive} measures, because we have used
Lemma~\ref{lem:technical}. Consequently, we do not know how
Eqs.~\eqref{eq:reco-eq} and \eqref{eq:gen-reco-coeff-DGL} are related
beyond this case. Fortunately, this is not required, as the next
result shows, where we use $\RR_{+} = \{ x\in\RR \mid x\geqslant 0\}$.

\begin{prop}\label{prop:reduction-props}
  The ODE system defined by Eq.~\eqref{eq:gen-reco-coeff-DGL}, with\/
  $\gamma$ as in Eq.~\eqref{eq:def-gamma}, is of dimension\/
  $B(n)=|\PP(S)|$, and has a unique solution for its Cauchy problem.
  
  The closed cone\/ $\cM_{+} (\PP(S)) \simeq \RR_{+}^{\nts B(n)}$ is
  invariant in forward time, and the\/ $\| .  \|^{\pa}_{1}$-norm of
  non-negative initial conditions is preserved. In particular, the
  simplex\/ $\cP(\PP(S))$ of probability vectors on\/ $\PP(S)$ is
  invariant in forward time.
\end{prop}

\begin{proof}
  The ODE system of Eq.~\eqref{eq:gen-reco-coeff-DGL} can be written
  as $\dot{a}^{\pa}_{t} = \varPsi (a^{\pa}_{t})$, where $\varPsi$ is
  nonlinear.  The first claim is clear, while solution uniqueness
  once again follows from Lipschitz continuity of $\varPsi$, which is
  obvious here.

Let $q\in\cM_{+} (\PP(S)) = \RR_{+}^{\PP (S)}\nts \simeq \RR_{+}^{B(n)}$ 
be arbitrary and consider any subset $\GG\subseteq\PP(S)$
such that $q(\GG) := \sum_{\cA\in\GG} q(\cA) = 0$. Then, one has
\[
    \bigl(\varPsi (q)\bigr) (\GG) \, = \sum_{\cA\in\GG}
     \bigl(\varPsi (q)\bigr) (\cA) \, = \, -\rtot{} \,
     q(\GG) \, + \sum_{\cA\in\GG}\, \sum_{\udo{\cB} \succcurlyeq \cA}
     \gamma (q; \cA,\cB\ts ) \ts \varrho(\cB\ts ) \, 
      \geqslant \, 0 \ts ,
\]
since $q(\GG)=0$ by assumption and all $\gamma (q; \cA,\cB\ts )$ as
well as all $\varrho (\cB\ts )$ are non-negative.  This gives the
invariance of $\cM_{+} (\PP(S))$ under the forward flow by standard
arguments \cite[Thm.~16.5]{Amann}.

Repeating the calculation for $\GG = \PP(S)$ yields
\[
    \bigl(\varPsi (q)\bigr) (\PP(S)) \, = \, -\rtot{} \,
    q(\PP(S)) \, + \sum_{\cB\in\ts\PP(S)} \varrho(\cB\ts )
    \sum_{\udo{\cA\ts}\nts \preccurlyeq \cB} \gamma(q;\cA,\cB\ts ) \ts .
\]
Now, the last sum (for any fixed $\cB$) is
\begin{equation}\label{eq:one-box}
    \sum_{\udo{\cA\ts}\nts \preccurlyeq \cB} 
     \gamma(q;\cA,\cB\ts ) \, = \, 
    \| q \|^{\pa}_{1} \, = \, q(\PP(S)) \ts ,
\end{equation}
as follows from Eq.~\eqref{eq:extend-gamma} by taking norms on both
sides and observing that $ \| R^{\pa}_{\nts \cA} (\nu) \| = \| \nu \|
= 1$ for all $\cA\preccurlyeq\cB$ as well as $ \| R^{\pa}_{\cB}
(\omega) \| = 1$, while one has $\gamma(q;\cA,\cB\ts )\geqslant 0$ for
any $q\in\cM_{+}(\PP(S))$; alternatively, one can also verify this
with a direct calculation on the basis of the definition of $\gamma$.
Consequently, $ \bigl(\varPsi (q)\bigr) (\PP(S)) =0 $, which implies
the claimed norm preservation. The positive invariance of
$\cP(\PP(S))$ is then clear.
\end{proof}

\begin{remark}\label{rem:reduction}
  Observing that $\gamma (q;\cA,\pmax) = q(\cA)$ holds for any
  $q\in\RR^{\PP (S)}$, one can rewrite the ODE of
  Eq.~\eqref{eq:reco-eq} as
\[
    \dot{a}^{\pa}_{t} (\cA) \, = \, - \bigl( 
        \rtot{} - \varrho(\pmax)\bigr)
    \ts a^{\pa}_{t}  (\cA) \, + \! \! 
         \sum_{\cA\preccurlyeq \udo{\cB} \prec \pmax}
    \!\! \gamma( a^{\pa}_{t} ; \cA,\cB ) \, \varrho (\cB) \ts ,
\]   
which corresponds to the observation made in Remark~\ref{rem:balance}.
\exend
\end{remark}

With Proposition~\ref{prop:reduction-props}, we may now conclude that
our ansatz \eqref{eq:gen-reco-coeff-DGL} is consistent and suitable
for the reduction of the original problem to a finite-dimensional one.

\begin{theorem}\label{thm:equivalence}
  The one-parameter family of measures\/ $\{ \omega^{\pa}_{t}\nts\nts
  \mid t\geqslant 0\}$ is a solution of the Cauchy problem of
  Eq.~\eqref{eq:reco-eq} with initial condition\/ $\omega^{\pa}_{0}
  \in \cM_{+} (X)$ if and only if it is of the form
\[
    \omega^{\pa}_{t} \, = \! \sum_{\cB\in\ts\PP(S)}\! \! 
    a^{\pa}_{t} (\cB\ts ) \, R^{\pa}_{\cB} (\omega^{\pa}_{0})
\]
  where the coefficient functions satisfy the Cauchy problem of
  Eq.~\eqref{eq:gen-reco-coeff-DGL} with initial condition\/
  $a^{\pa}_{0} (\cB\ts ) = \delta (\cB,\pmax)$.
\end{theorem}
 
\begin{proof}
  The solution property is clear from our above calculations, while
  the correspondence of the initial conditions is obvious for the
  generic case, and extends to the general case by a standard
  continuity argument.  The claim now follows from the uniqueness
  statements for the two Cauchy problems.
\end{proof}

\begin{remark}
  Let us mention that an alternative path to
  Theorem~\ref{thm:equivalence} is possible via \cite[Thm.~16.5 and
  Rem.~16.6]{Amann}, by showing that the convex set $\varDelta_{\nu}$
  is forward invariant for any $\nu\in\cM_{+} (X)$ as initial
  condition. This requires the verification of the `inside reflection
  property' for any piece of the boundary of $\varDelta_{\ts \nu}$,
  which is somewhat tedious in view of the possible degeneracies. This
  is the reason why we opted for our approach above.  \exend
\end{remark}

Let us briefly discuss the structure of $\varDelta_{\ts \nu}$ for a
general $\nu\in\cM_{+} (X)$, including the non-generic cases. Each
$\nu$ gives rise to a set of partitions
\[
      \HH_{\nu} \, := \, \{ \cA \in \PP (S) \mid 
      R^{\pa}_{\nts \cA} (\nu) = \nu \}\ts ,
\]
which is non-empty (since $\pmax \in \HH_{\nu}$ for all $\nu$) and
defines a unique partition
\[
    \cU_{\nu} \, := \, \bigwedge \! \HH_{\nu} \, = \,
    \{ U^{\pa}_{1}, \ldots , U^{\pa}_{r} \} \ts ,
\]
so that $\HH_{\nu} = \{ \cA\in\PP (S) \mid \cA\succcurlyeq \cU_{\nu}
\}$.  The partition $\cU_{\nu}$ also defines the sublattice $\PP_{\nu}
:= \PP (U^{\pa}_{1}) \times \ldots \times \PP (U^{\pa}_{r})$ of
product form, with $B(\lvert U^{\pa}_{1}\rvert ) \cdot \ldots \cdot
B(\lvert U^{\pa}_{r}\rvert )$ elements. It is precisely this
sublattice of $\PP (S)$ that determines the structure of the convex
set $\varDelta_{\ts \nu}$, which is now the Cartesian product of $r$
simplices, and of total dimension $\bigl(B(\lvert U^{\pa}_{1}\rvert
)-1\bigr) + \ldots + \bigl(B(\lvert U^{\pa}_{r}\rvert ) - 1\bigr)$.

It is clear that the time evolution of a measure $\nu\in\cM_{+} (X)$
under the flow of the recombination equation \eqref{eq:reco-eq} can
thus be reduced to a smaller ODE system, which is fully consistent
with our above treatment as a consequence of
Eq.~\eqref{eq:product-structure}, as $\varPhi (\nu) = \varPhi \bigl(
R^{\pa}_{\cU_{\nu}} \nts (\nu) \bigr)$ then means that effectively
only partitions from the set $\PP_{\nu}$ are involved. We leave it to
the reader to spell out the details for the modified correspondence
and the appropriate initial conditions for the finite-dimensional ODE
system.

We are now in the position to approach a solution of the recombination
equation.

\section{Solution under a linearity 
assumption}\label{sec:linear}

The relatively simple solution structure in the case of
single-crossover dynamics (see references \cite{BB,MB}) was due to the
fact that the nonlinear recombinators acted \emph{linearly} along
solutions. Since this is a special case of our general model, via
setting $\varrho (\cA) = 0$ for any $\cA$ that fails to be an ordered
partition with (at most) two parts, it is reasonable to consider this
point of view also more generally. It will turn out that the linearity
assumption is false in general, but we can still learn some
interesting things along the way.

Thus, let us assume that also the more general
recombinators act linearly along the solution. A simple calculation
shows that our coefficient functions then have to satisfy the ODEs
\begin{equation}\label{eq:def-lineq}
  \frac{\dd}{\dd t}\,  \alin (\cA) \, = \, - \rtot{}\,
  \alin (\cA)\, + \sum_{\udo{\cB} \succcurlyeq \cA}
  \beta (\alin ; \cA,\cB\ts ) \ts \varrho (\cB\ts )
\end{equation}
for all $\cA\in\ts\PP(S)$, where, for any fixed $q\in\cM(\PP(S))$,
\begin{equation}\label{eq:two-box}
   \beta (q ; \cA,\cB\ts ) \, := \begin{cases}
   \sum_{ \udo{\cC} \wedge\ts \cB = \cA}\, q(\cC) ,
     & \text{if $\cA \preccurlyeq \cB$}, \\
   0 , & \text{otherwise}, \end{cases}
\end{equation}
is another element of the incidence algebra $\AAA(\PP(S))$.  The
difference of Eq.~\eqref{eq:def-lineq} to the general equation
\eqref{eq:gen-reco-coeff-DGL} thus lies in the replacement of the
function $\gamma$ by the significantly simpler linear function
$\beta$.

\begin{prop}\label{prop:lin-solve}
  The Cauchy problem defined by Eq.~\eqref{eq:def-lineq} together with
  the initial condition\/ $a^{\mathrm{lin}}_{0} (\cA) = \delta
  (\cA,\pmax) $ has the unique solution given by
\[
   \alin (\cA) \, = 
   \sum_{\substack{\GG \subseteq \PP(S) \\ \bigwedge\! \GG = \cA}} \,
   \prod_{\sigma\in\GG} \bigl(1- \ee^{-\varrho(\sigma) t} \bigr)
   \prod_{\tau\in\GG_{\phantom{I}}^{\mathsf{c}}}\! \ee^{-\varrho(\tau) t} ,
\]
  which, for $t\geqslant 0$,  constitutes a one-parameter family of
  probability vectors on\/ $\PP(S)$.
\end{prop}

\begin{proof}
  Consider the (upper) summatory function $\bigl(F\alin\bigr) (\cA) :=
  \sum_{\cB \succcurlyeq \cA} \alin (\cB\ts )$, which is
\[
\begin{split}
   \bigl(F\alin\bigr) (\cA) \, & = 
   \sum_{\substack{\GG \subseteq \PP(S) \\ \bigwedge\! \GG \succcurlyeq \cA}} \,
   \prod_{\sigma\in\GG} \bigl(1- \ee^{-\varrho(\sigma) t} \bigr)
   \prod_{\tau\in\GG_{\phantom{I}}^{\mathsf{c}}}\! \ee^{-\varrho(\tau) t} \\[1mm]
   & = \prod_{\cB \notin [\cA,\pmax]} \ee^{-\varrho (\cB\ts ) t}
   \sum_{\GG \subseteq [\cA,\pmax]} \,
   \prod_{\sigma\in\GG} \bigl(1- \ee^{-\varrho(\sigma) t} \bigr)
   \prod_{\tau\in\GG_{\phantom{I}}^{\mathsf{c}}\cap [\cA,\pmax]} 
   \! \! \ee^{-\varrho(\tau) t} \\[2mm]
   & =  \prod_{\cB \notin [\cA,\pmax]} \ee^{-\varrho (\cB\ts ) t}
   \prod_{\cC \in [\cA,\pmax]} 
   \bigl(1- \ee^{-\varrho(\cC) t} + \ee^{-\varrho (\cC) t} \bigr)
   \, =  \prod_{\cB \notin [\cA,\pmax]} \ee^{-\varrho (\cB\ts ) t}
\end{split}
\] 
and hence satisfies the simple ODE
\begin{equation}\label{eq:summatory-DGL}
   \frac{\dd}{\dd t}\ts  \bigl(F\alin\bigr) (\cA) \, = \,
   - \biggl(\,\sum_{\cB \notin [\cA,\pmax]} \! \varrho (\cB\ts )\! \biggr)
   \, \bigl(F\alin\bigr) (\cA) \ts ,
\end{equation}
together with the initial condition $ \bigl(F a^{\mathrm{lin}}_{0}
\bigr) (\cA) = 1$ for all $\cA\in\ts\PP(S)$.

On the other hand, for fixed $\cA,\cB\in\ts\PP (S)$ with
$\cA \preccurlyeq \cB$, we have
\[
   \sum_{\cC \in [\cA,\cB]} \beta (\alin; \cC, \cB\ts ) \, = \!
   \sum_{\cC \in [\cA,\cB]} \, \sum_{\udo{\cD} \wedge \cB = \cC}
   \! \alin (\cD) \, = \! \sum_{ \udo{\cD} \wedge \cB \succcurlyeq \cA} 
    \! \alin (\cD) \, = \sum_{\udo{\cD}\succcurlyeq \cA}
   \alin (\cD) \, = \, \bigl( F \alin\bigr) (\cA) \ts ,
\]
wherefore the summatory function of the right-hand side of
Eq.~\eqref{eq:def-lineq}, evaluated at $\cA$, becomes
\[
   - \rtot{} \bigl( F \alin \bigr) (\cA) + \biggl(\,
     \sum_{\udo{\cB}\succcurlyeq \cA} \varrho (\cB\ts )\! \biggr) 
    \bigl(F \alin\bigr) (\cA)
  \, = \, - \biggl(\,\sum_{\cB \notin [\cA,\pmax]} 
    \!  \varrho (\cB\ts ) \! \biggr)
     \, \bigl(F\alin\bigr) (\cA) \ts ,
\]
which agrees with the right-hand side of Eq.~\eqref{eq:summatory-DGL}.
The claim now follows from (upper) M\"{o}bius inversion, because, for
all $t\geqslant 0$ and for all $\cA\in\ts\PP(S)$, one has
\begin{equation}\label{eq:Moebius-sum}
    \alin (\cA) \, =  \sum_{\udo{\cB}\succcurlyeq\cA} 
    \mu(\cA,\cB\ts )\ts  \bigl(F\alin\bigr) (\cB\ts )\ts ,
\end{equation}
where $\mu\in\AAA(\PP(S))$ denotes the M\"{o}bius function for the
lattice $\PP(S)$.

The fact that $a^{\pa}_{t}$, for each $t\geqslant 0$, is a probability
vector on $\PP(S)$ follows from our earlier calculation in
Eq.~\eqref{eq:c-def}, with $f(x) = \ee^{x}$.
\end{proof}

This approach clearly has a lattice-theoretic basis, which lends
itself to a number of interesting further aspects and insights
\cite{BS}.

\begin{remark}\label{rem:lin-decay-inverse}
  Let us note that Eqs.~\eqref{eq:summatory-DGL} and
  \eqref{eq:Moebius-sum} allow a re-interpretation of the coefficient
  formula from Proposition~\ref{prop:lin-solve} as
\begin{equation}\label{eq:lin-decay}
   \alin (\cA) \, =  \sum_{\udo{\cB}\succcurlyeq \cA}
   \mu(\cA,\cB\ts ) \, \ee^{-\chi (\cB\ts ) t}
\end{equation}
with $\chi (\cB\ts ) = \sum_{\cC\notin [\cB,\pmax]} \varrho (\cC)$
being the \emph{decay rate} of the corresponding (exponential)
term. This relation means $\chi(\cB) = \rtot{} - \sum_{\udo{\cC}
  \succcurlyeq\cB} \varrho(\cB)$, so that the recombination rates are
obtained from the decay rates by means of (upper) M\"{o}bius inversion
as
\[
    \varrho (\cB) \, = \, \delta(\cB,\pmax) \ts \rtot{} \, - 
    \sum_{\udo{\cC} \succcurlyeq \cB} \mu(\cB,\cC)
    \, \chi(\cC) \ts ,    
\]
where it is assumed that the total recombination rate $\rtot{}$ is
known. This detail corresponds to the fact that $\varrho (\pmax)$
does not contribute to any of the $\chi (\cA)$.

Either version of $\alin$ permits the determination of the asymptotic
properties of the coefficients $\alin$ as $t\to\infty$, and hence that
of the measure $\omega^{\pa}_{t}$. In particular, when the partition
subset $\GG := \{ \cA \in \PP(S) \mid \varrho (\cA) > 0 \}$ satisfies
$\bigwedge\!\GG = \pmin$, one has
\[
    \lim_{t\to\infty} \omega^{\pa}_{t} \, = \, 
    R^{\pa}_{\pmin} (\omega^{\pa}_{0}) \ts ,
\]
with convergence in the $\|.\|$-topology. This is the obvious
generalisation of the known asymptotic properties in the special
cases treated in \cite{Buerger,BB}.
\exend
\end{remark}

\begin{remark}
  Let us briefly mention that the linear system of ODEs defined by
  Eq.~\eqref{eq:def-lineq} can still be solved when the recombination
  rates become time-dependent, as was previously observed in
  \cite[Addendum]{MB} for single-crossover recombination.  Indeed, if
  all $\varrho^{\pa}_{t} (\cA)$ are non-negative functions of time
  (which is needed to preserve all claims of
  Proposition~\ref{prop:reduction-props} also for
  Eq.~\eqref{eq:def-lineq} with time-dependent rates), the solution
  formula from Proposition~\ref{prop:lin-solve} becomes
\[
    \alin (\cA) \, = 
   \sum_{\substack{\GG \subseteq \PP(S) \\ \bigwedge\! \GG = \cA}} \,
   \prod_{\sigma\in\GG} \, \Bigl(1- \exp 
    \bigl(- {\textstyle \int_{0}^{t} }
    \, \varrho^{\pa}_{s}(\sigma) \dd s \bigr) \Bigr)
   \prod_{\tau\in\GG_{\phantom{I}}^{\mathsf{c}}}\! \exp \bigl(
    - {\textstyle \int_{0}^{t}}\, \varrho^{\pa}_{s} (\tau) 
    \dd s \bigr) \ts ,
\]
for the same initial conditions.  The proof is completely analogous to
that of Proposition~\ref{prop:lin-solve}, now with
\[
   \bigl( F \alin (\cA)\bigr) \, = \! \prod_{\cB\notin [\cA,\pmax]}
    \! \! \exp\bigl( - {\textstyle \int_{0}^{t}}\, 
   \varrho^{\pa}_{s} (\cB) \dd s \bigr) .
\]
This summatory function then satisfies the ODE
\[
    \frac{\dd}{\dd t}\ts  \bigl(F\alin\bigr) (\cA) \, = \,
   - \biggl(\,\sum_{\cB \notin [\cA,\pmax]} \! \! \varrho^{\pa}_{t} 
   (\cB\ts )\! \biggr) \, \bigl(F\alin\bigr) (\cA) \ts ,
\]
which replaces Eq.~\eqref{eq:summatory-DGL} in this generalisation.
\exend
\end{remark}

\begin{theorem}\label{thm:better-than-nothing}
  Let\/ $\cA\in\ts\PP(S)$ be\/ $\cA = \pmax$ or a partition into
  two parts, one of which is a singleton set. Then, the coefficient
  formula from Proposition~$\ref{prop:lin-solve}$ gives the correct
  solution also for the Cauchy problem of the general recombination
  equation from Eq.~\eqref{eq:gen-reco-coeff-DGL}.
\end{theorem}

\begin{proof}
  Let a probability vector $q$ on $\PP(S)$ be given.  The definitions
  of $\beta$ and $\gamma$ from Eqs.~\eqref{eq:one-box} and 
  \eqref{eq:two-box} imply, via a simple calculation, that
\[
    \gamma(q;\cA,\pmax) \, = \, q (\cA) \, = \,
    \beta(q;\cA,\pmax)
\]
holds for any $\cA\in\ts\PP(S)$. This gives the claim for $\cA =
\pmax$, because Eqs.~\eqref{eq:gen-reco-coeff-DGL} and
\eqref{eq:def-lineq} are equal in this case, hence 
$a^{\pa}_{t} (\pmax) = \alin (\pmax)$.

More generally, when $\cA = \bigl\{ \{i\}, S\nts\setminus\nts\nts
\{i\}\bigr\}$ for some $i\in S$, one finds that
\[
   \gamma(q;\cA,\cA) \, = \, q(\pmax) + q(\cA) \, = \,
   \beta(q;\cA,\cA) \ts ,
\]
so that the ODEs from Eqs.~\eqref{eq:gen-reco-coeff-DGL} and
\eqref{eq:def-lineq} coincide also for such partitions $\cA$, which
proves the second claim.
\end{proof}

\begin{coro}\label{coro:add-on}
  The coefficient formula from Proposition~$\ref{prop:lin-solve}$
  gives the correct solution of Eq.~\eqref{eq:gen-reco-coeff-DGL} for
  all\/ $\cA\in\ts\PP (S)$ whenever\/ $S$ is a finite set
  with\/ $1 \leqslant |S| \leqslant 3$.
\end{coro}

\begin{proof}
  There is nothing to prove for $|S|=1$.  The claim for $|S|=2$ is
  obvious, and also follows from \cite[Prop.~3]{MB}. When $|S|=3$, all
  partitions except $\cA=\pmin$ satisfy the conditions of
  Theorem~\ref{thm:better-than-nothing}.  Since $a^{\pa}_{t} (\pmin) =
  1 - \sum_{\cB \ne \pmin} a^{\pa}_{t} (\cB\ts )$, the claim follows.
\end{proof}

\begin{remark}\label{rem:old-case}
  In the special situation of single-crossover recombination, where
  $\varrho (\cA) > 0$ only for \emph{ordered} partitions $\cA$ into
  two parts, the solution formula of Proposition~\ref{prop:lin-solve}
  reduces to the known solution for this case from \cite{BB,MB}.  In
  particular, the linearity assumption is satisfied, and the solution
  holds for all system sizes and all values of the single-crossover
  rates.  \exend
\end{remark}

Note, however, that already for $|S|=4$, when we are beyond the
single-crossover case, the coefficients $a^{\pa}_{t}$ and $\alin$ can
differ, for instance for $\cA = \{ \{1,2\},\{3,4\}\}$, which is a
biologically relevant partition. We thus need to proceed without the
linearisation assumption.

\section{General case: Marginalisation consistency}
\label{sec:marginal}

It is clear that our general recombination equation can only be
considered a reasonable model if it is \emph{marginalisation
  consistent}. By this we mean that the restriction to a subsystem,
via appropriate marginalisation, gives a solution of the recombination
equation for the subsystem.  We now discuss this in more detail, and
then establish this consistency property for our model, both in the
measure-theoretic and in the finite-dimensional version. The latter
case will depend on an interesting interplay between elements of the
M\"{o}bius and the incidence algebras at hand.

Let $S$ be as above, or any other finite set with $n$ elements, and
consider a subsystem as specified by $\varnothing \ne U\nts \subseteq
S$. When $\omega^{S}_{t}$ is the solution of the general recombination
equation \eqref{eq:reco-eq} with initial condition $\omega^{S}_{0} \in
\cM_{+} (X)$ according to Proposition~\ref{prop:gensol}, it is natural
to define
\begin{equation}\label{eq:red-measure}
    \omega^{U}_{t} = \, \pi^{\pa}_{U}\nts . \ts \omega^{S}_{t}
\end{equation}
as the corresponding (marginalised) measure for the subsystem defined
by $U$. Then, recalling that the projector $\pi^{\pa}_{U}$ is linear,
we get
\[
\begin{split}
   \frac{\dd}{\dd t}\, \omega^{U}_{t} & = \,
   \pi^{\pa}_{U}\nts . \! \left(\! \frac{\dd}{\dd t}\, 
   \omega^{S}_{t} \right)\, = \, \pi^{\pa}_{U} . 
    \biggl( \sum_{\,\cB\in\PP (S)}\! \varrho^{\ts S} (\cB)\,
    \bigl( R^{\ts S}_{\cB} - \pmax \bigr) 
     (\omega^{S}_{t}\ts ) \biggr) \\[1mm]
    & = \!  \sum_{\cB\in\PP (S)} \!  \varrho^{\ts S} (\cB) \,
    \Bigl(\nts \pi^{\pa}_{U} . \bigl( R^{\ts S}_{\cB}
    (\omega^{S}_{t}\ts )\bigr) -  \omega^{U}_{t} \Bigr) 
    \, = \! \sum_{\cB\in\PP (S)} \!  \varrho^{\ts S} (\cB) \,
    \Bigl(\nts\nts  R^{\ts U}_{\cB |^{\pa}_{U}}\nts
    ( \omega^{U}_{t}) -  \omega^{U}_{t} \Bigr) \\[2mm]
    & = \! \sum_{\cA\in\PP (U)} \,
    \sum_{\substack{\cB\in\PP (S) \\ \cB |^{\pa}_{U} = \cA}}
    \! \varrho^{S} (\cB) \; \Bigl(\nts  R^{\ts U}_{\cA}
    ( \omega^{U}_{t}) -  \omega^{U}_{t} \Bigr)  \, = 
    \! \sum_{\cA\in\PP (U)} \! \varrho^{U} \! (\cA) \,
    \bigl( R^{\ts U}_{\nts \cA} - \pmax \bigr) 
    (\omega^{U}_{t}) \ts ,
\end{split}
\]
where the fourth step is an application of Lemma~\ref{lem:technical},
while the last step anticipates the definition of the \emph{induced}
(or \emph{marginal}\ts)  
\emph{recombination rates} for the subsystem as
\begin{equation}\label{eq:def-induced-rates}
  \varrho^{\ts U}\! (\cA) \, := \!
  \sum_{\substack{\cD \in \PP(S) \\ \cD |^{\pa}_{U} = \cA}} 
  \! \varrho^{S} (\cD) 
\end{equation}
for any $\cA\in\PP (U)$. It is obvious that non-negativity of the
rates $\varrho^{S}$ implies that of the marginal rates $\varrho^{U}\!$.
Our little calculation proves the following result.

\begin{prop}\label{prop:measure-subsystem}
  Let\/ $S$ be a finite set and\/ $U\nts\subset S$ a non-empty subset.
  If\/ $\omega^{S}_{t}$ is a solution of the recombination equation
  \eqref{eq:reco-eq}, with recombination rates\/ $\varrho^{S} (\cB)
  \geqslant 0$ for\/ $\cB\in \PP (S)$ and with initial condition\/
  $\omega^{S}_{0} \in \cM_{+} (X)$, the marginalised measure\/
  $\omega^{U}_{t}$ from Eq.~\eqref{eq:red-measure} solves the
  recombination equation for the subsystem defined by\/ $U\! $,
  provided the recombination rates\/ $\varrho^{U}\! (\cA)$ for\/
  $\cA\in \PP (U)$ are defined according to
  Eq.~\eqref{eq:def-induced-rates}.  \qed
\end{prop}

\begin{remark}
  The problem of marginalisation consistency has been observed early
  on in mathematical population genetics. Ewens and Thomson \cite{ET}
  have tackled it in 1977 in models that describe the combined action
  of recombination \emph{and selection} for the two-parent case in
  discrete time; see also the review in
  \cite[pp.~69--72]{Buerger}. The corresponding dynamics is, in
  general, \emph{not} marginalisation consistent. However, it is
  obvious from the calculations in \cite{ET} that consistency does
  apply in the case without selection, and the dynamics is then
  governed by the marginal recombination probabilities, which are the
  two-parent, discrete-time analogues of our marginal recombination
  rates.  \exend
\end{remark}

Let us see how this result translates to the finite-dimensional ODE
systems at the level of the coefficient functions $a^{U}_{t}\! (\cA)$.
Given a general probability vector $q=q^{S}$ on $\PP(S)$, we define
the corresponding \emph{marginal probabilities} via
\begin{equation}\label{eq:def-marginal}
   q^{\ts U} \! (\cA) \, := 
   \sum_{\substack{\cD \in \PP (S)\\ \cD|^{\pa}_{\nts U} = \cA }}
   \! q^{S} (\cD) \ts ,
\end{equation}
for any partition $\cA \in \PP (U)$, in complete analogy to
Eq.~\eqref{eq:def-induced-rates}.  
Clearly, one has $q^{\ts U} \!
(\cA)\geqslant 0$, while proper normalisation follows from
\begin{equation}\label{eq:marg-norm}
   \sum_{\cA \in \PP (U)} \! q^{\ts U} \! (\cA) \,  = 
   \sum_{\cA \in \PP (U)}  \,
   \sum_{\substack{\cD \in \PP (S)\\ \cD|^{\pa}_{\nts U} = \cA }} 
   \! q^{S} (\cD) \, =  \sum_{\cB \in \PP (S)}  \,
   \sum_{\substack{\cC\in \PP (U) \\ \cC = \cB |^{\pa}_{\nts U} }}
    \! q^{S} (\cB\ts )  \, =   \sum_{\cB \in \PP (S)} \!
     q^{S} (\cB\ts ) \, = \, 1 \ts ,
\end{equation} 
where the penultimate step follows from the observation that the inner
sum over $\cC$ consists of precisely one term because the restriction
of a partition to a subset $U\nts\subseteq S$ is unique.

Let us note for later use that, for $\varnothing \ne U \subseteq V
\subseteq S$ and $\cA \in \PP (U)$, we also have
\begin{equation}\label{eq:marginal-gen}
   q^{\ts U} \! (\cA) \, = 
   \sum_{\substack{\cE \in \PP (V)\\ \cE|^{\pa}_{\nts U} = \cA }}
   \! q^{V} \! (\cE) \ts ,
\end{equation}
because 
\[
  \sum_{\substack{\cE \in \PP (V)\\ \cE|^{\pa}_{\nts U} = \cA }}
   \! q^{V}\! (\cE) \, = 
  \sum_{\substack{\cE \in \PP (V)\\ \cE|^{\pa}_{\nts U} = \cA }}\,
  \sum_{\substack{\cD \in \PP (S)\\ \cD|^{\pa}_{\nts V} = \cE }}
   \! q^{S} (\cD) \; = 
  \sum_{\substack{\cD \in \PP (S)\\ \cD|^{\pa}_{\nts U} = \cA }}
   \! q^{S} (\cD) \, = \, q^{U} \! (\cA) \,.
\]
As special cases of Eqs.~\eqref{eq:def-marginal} and
\eqref{eq:marginal-gen}, let us also note that, for $\cB \in \PP(S)$
and $\varnothing \ne U \subseteq V \subseteq S$, one has
\begin{equation}\label{eq:marginal-restricted}
   q^{\ts U} \! (\cB|^{\pa}_{\nts U}) \; = \!
   \sum_{\substack{\cD \in \PP (S)\\ \cD|^{\pa}_{\nts U} = \cB|^{\pa}_{\nts U} }}
   \! \! q^{S} (\cD) \; = \!
   \sum_{\substack{\cE \in \PP (V)\\ \cE\nts|^{\pa}_{\nts U} = \cB|^{\pa}_{\nts U} }}
   \! q^{V} \! (\cE) \ts .
\end{equation}

\begin{remark}\label{rem:marg-sum}
   The (upper) summatory function defined by $(Fq^{U}) (\cA) =
   \sum_{\udo{\cB}\succcurlyeq\cA} q^{U} \!(\cB)$ can also be calculated
   by marginalisation, namely as
\[
 \begin{split}
   \bigl(Fq^{U}\bigr) (\cA) \, & = 
    \sum_{\udo{\cB}\succcurlyeq\cA} q^{U}\! (\cB) \, =
   \sum_{\udo{\cB}\succcurlyeq\cA} \,
   \sum_{\substack{\cD \in \PP (S)\\ \cD|^{\pa}_{\nts U} = \cB }} 
   q^{S}\nts (\cD) \, = \!
   \sum_{\substack{\cD \in \PP (S)\\ \cD|^{\pa}_{\nts U} \succcurlyeq \cA }}
   \! q^{S}\nts (\cD) \\ & = \!
   \sum_{\substack{\cC \in \PP (S)\\ \cC|^{\pa}_{\nts U}= \cA }}
   \, \sum_{\udo{\cD}\succcurlyeq \cC}  q^{S}\nts (\cD) \, =
   \sum_{\substack{\cC \in \PP (S)\\ \cC|^{\pa}_{\nts U}= \cA }}
   \! \bigl(F q^{S} \bigr) (\cC) \ts ,
 \end{split}
\]
   wherefore consistency at this level is obvious.
\exend
\end{remark}

Let us return to the recombination rates $\varrho^{S} (\cB\ts )$ for
$\cB\in\ts\PP(S)$ together with
\[
    \rtot{S} \, = \!
    \sum_{\cB \in \PP (S)} \!\!  \varrho^{S} (\cB\ts ) \ts ,
\]
compare Remark~\ref{rem:balance}, and consider the marginal rates
$\varrho^{U}\! (\cA)$ with $\cA\in\PP (U)$ and $\varnothing \ne U \nts
\subseteq S$.  Repeating the above calculation of the normalisation
condition, one finds
\begin{equation}\label{eq:rtot}
     \rtot{U} \, = \, \rtot{S}  \, = \, \rtot{} \ts ,
\end{equation}
as it should be. The total recombination rate is thus the same on all
levels, and independent of $U$, as it must.  Note that, in this
process, $\varrho^{\ts U}\! (\{ U \} ) = \varrho^{U} \! (\pmax)$ need
not vanish. This does not matter because the corresponding
recombinator (on the subsystem) is the identity, and hence does not
affect the solution; compare Remark~\ref{rem:balance}. \medskip

Let us now see how the result of
Proposition~\ref{prop:measure-subsystem} translates to properties of
the coefficients for the ODE system and its subsystems. Here, the
desired marginalisation consistency will depend on the following
slightly technical, but somewhat surprising identity.
\begin{lemma}\label{lem:marginal}
  Let\/ $U\nts\subseteq S$ be as before, and let\/ $\cA,\cB \in \PP(U)$
  with\/ $\cA \preccurlyeq \cB$ be arbitrary, but fixed. Then, for
  any\/ $\cD\in\ts\PP(S)$ with\/ $\cD |^{\pa}_{U} = \cB$, one has the
  product representation and reduction relation
\[
   \sum_{\substack{\PP(S) \owns\ts \udo{\cC}
          \preccurlyeq \cD \\ \cC |^{\pa}_{\nts U} = \cA}}
   \gamma^{S} (q^{S}; \cC,\cD) \, = \!
   \sum_{\substack{\PP(S) \owns\ts \udo{\cC}
          \preccurlyeq \cD \\ \cC |^{\pa}_{\nts U} = \cA}}\,
   \prod_{i=1}^{|\cD|} q^{\ts D_i}  (\cC|^{\pa}_{\nts D_i})  \, = \,
    \prod_{i=1}^{|\cB|} q^{\ts B_i}  (\cA|^{\pa}_{\nts B_i})  \, = \,
   \gamma^{U}\nts\nts (q^{\ts U}\! ; \cA,\cB\ts ) \ts ,
\]
where\/ $q^{S}$ is any probability vector on\/ $\PP(S)$, $q^{\ts U}\!$
its marginalisation according to Eq.~\eqref{eq:def-marginal}, and\/
$\gamma^{U}\!$ is defined as in Eq.~\eqref{eq:def-gamma}, but with\/
$S$ replaced by\/ $U$.
\end{lemma}

\begin{proof}
  Let $\cA,\cB \in \PP(U)$ with $\cA \preccurlyeq \cB$ be given, and
  assume $\cB = \{ B_{1}, \dots , B_{r}\}$.  With our definition of
  the $\gamma$-function in Eq.~\eqref{eq:def-gamma} together with
  Eqs.~\eqref{eq:marginal-restricted} and \eqref{eq:def-marginal},
  the right-hand side evaluates as
\begin{equation}\label{eq:gamma-U}
   \gamma^{U}\! (q^{\ts U}\! ; \cA,\cB\ts ) \, = \, \prod_{i=1}^{r}
   \sum_{\substack{\cC\in\ts\PP(U) \\
    \cC|^{\pa}_{\nts B_i}\! = \cA|^{\pa}_{\nts B_i}}}
   \!\!\! q^{\ts U}\! (\cC) \, = \, 
   \prod_{i=1}^{r}  q^{\ts B_i}  (\cA|^{\pa}_{\nts B_i}) \ts .
\end{equation}
For an arbitrary $\cD\in\ts\PP(S)$ with $\cD = \{D_1, \ldots,
D_{|\cD|}\}$ and $\cD|^{\pa}_{U} = \cB$, this is now to be compared
with
\begin{equation}\label{eq:intermediate}
   \sum_{\substack{\PP(S) \owns\ts  \udo{\cC}
           \preccurlyeq \cD \\ \cC |^{\pa}_{\nts U} = \cA}}
   \! \! \gamma^{S} (q^{S}; \cC,\cD) \; = 
   \sum_{\substack{\PP(S) \owns\ts  \udo{\cC}
           \preccurlyeq \cD \\ \cC |^{\pa}_{\nts U} = \cA}}\,
   \prod_{i=1}^{|\cD|}  
   \sum_{\substack{\cE\in\ts\PP(S) \\ \cE|^{\pa}_{\nts D_i}\! 
        = \cC|^{\pa}_{\nts D_i}}}
   \!\!\! q^{\ts S}\! (\cE) \; = 
   \sum_{\substack{\PP(S) \owns\ts \udo{\cC}
           \preccurlyeq \cD \\ \cC |^{\pa}_{\nts U} = \cA}}
   \; \prod_{i=1}^{|\cD|} q^{\ts D_i}  (\cC|^{\pa}_{\nts D_i}) \ts ,
\end{equation}
where the last step is once again a consequence of
Eq.~\eqref{eq:def-marginal}.  To proceed, we now rewrite the partition
$\cD$ as $\cD = \{ D^{\prime}_{1},\dots, D^{\prime}_{r},
D^{\prime\prime}_{1},\dots ,D^{\prime\prime}_{s}\}$ with $r=|\cB|$ and
$D^\prime_{i} \cap U = B_{i}$ for $1\leqslant i\leqslant r$ (which is
without loss of generality) and $D^{\prime\prime}_{j} \cap U =
\varnothing$ for $1\leqslant j\leqslant s$. Note that $s=0$ is
possible, in which case no $D^{\prime\prime}_{j}$ is present and $r =
\lvert \cD \rvert$.

Now, any partition $\cC\in\ts\PP(S)$ in the summation in
Eq.~\eqref{eq:intermediate} must be a joining of the form
\[
   \cC \, = \, \cC^\prime_{1} \sqcup \cC^\prime_{2} \sqcup \dots
   \sqcup \cC^\prime_{r} \sqcup \cC^{\prime\prime}_{1} \sqcup 
   \dots \sqcup \cC^{\prime\prime}_{s}
\]
with $\cC^\prime_{i} \in \PP (D^\prime_{i})$ and
$\cC^{\prime\prime}_{j} \in \PP (D^{\prime\prime}_{j})$, subject to
the additional condition that we always have $\cC |^{\pa}_{U} = \cA$,
which means $\cC^\prime|^{\pa}_{\nts B_i} = \cA|^{\pa}_{\nts B_i}$ for
all $1\leqslant i\leqslant r$. The summation on the right-hand side of
Eq.~\eqref{eq:intermediate} can now be broken into smaller sums that
can be absorbed into the factors, which amounts to refining each block
of $\cD$ individually. This turns the right-hand side of
Eq.~\eqref{eq:intermediate} into a product of two terms, namely
\[
   \biggl(\,\prod_{i=1}^{r} \sum_{\substack{\sigma\in\ts\PP(D^\prime_i) \\
   \sigma|^{\pa}_{\nts B_i}\! = \cA|^{\pa}_{\nts B_i}}} q^{D^\prime_i} (\sigma)
   \biggr) \biggl(\, \prod_{j=1}^{s}\,
   \sum_{\tau \in \PP(D^{\prime\prime}_{j})}\,
   q^{D^{\prime\prime}_{j}} (\tau)\biggr).
\]
Now, the sum in each factor of the second product is clearly $1$
because $q^{D^{\prime\prime}_{j}}$ is a probability vector on
$\PP(D^{\prime\prime}_{j})$, compare~Eq.~\eqref{eq:marg-norm}.
Consequently, the entire second term is $1$, which is also true if
$s=0$ (in which case we have the empty product here).  Likewise, the
sum in the $i$th factor of the first product equals $q^{\ts B_i}
(\cA|^{\pa}_{\nts B_i})$ by Eq.~\eqref{eq:marginal-restricted},
because $B_{i} \subseteq D^\prime_{i}$ by our assumptions. We thus get
\[
   \sum_{\substack{\PP(S) \owns\ts  \udo{\cC}
           \preccurlyeq \cD \\ \cC |^{\pa}_{\nts U} = \cA}}
   \! \! \gamma^{S} (q^{S}; \cC,\cD) \, = \,
   \prod_{i=1}^{r}
    q^{\ts B_i}  (\cA|^{\pa}_{\nts B_i}).
\]
Together with Eqs.~\eqref{eq:gamma-U} and \eqref{eq:intermediate},
this proves the lemma.
\end{proof}

The marginalisation consistency can now be stated as follows.
\begin{prop}\label{prop:marg-consistent}
  Let\/ $S$ be a finite set and\/ $U\nts \subset S$ a non-empty
  subset. If the family of probability vectors\/ $\{ a^{S}_{t} \nts
  \mid t\geqslant 0\}$ is a solution of the Cauchy problem of
  Eq.~\eqref{eq:gen-reco-coeff-DGL} with initial condition\/
  $a^{S}_{0} (\cB\ts ) = \delta^{S}\nts (\cB,\pmax)$, the marginalised
  family\/ $\{ a^{U}_{t} \nts\nts \mid t\geqslant 0\}$, with\/
  $a^{U}_{t}\!  (\cA)$ defined according to
  Eq.~\eqref{eq:def-marginal} for all\/ $\cA\in\ts\PP(U)$, solves the
  corresponding Cauchy problem for the subsystem, with initial
  condition\/ $a^{U}_{0}\! (\cA) = \delta^{U}\!  (\cA,\pmax)$ and the
  marginal recombination rates of
  Eq.~\eqref{eq:def-induced-rates}. Explicitly, it satisfies the ODE
  \begin{equation}\label{eq:a-marg}
     \dot{a}^{\ts U}_{t} \nts (\cA) \;  = \; - \bigl( 
        \rtot{} - \varrho(\pmax)\bigr)
    \ts a^{\ts U}_{t}  (\cA) \, + \! \! 
         \sum_{\cA\preccurlyeq \udo{\cB} \prec \pmax}
    \!\!  \varrho^{\ts U} (\cB) 
    \prod_{i=1}^{|\cB|} a^{B_i}_{t} (\cA|^{\pa}_{\nts B_i})\ts . 
  \end{equation}
  The analogous statement remains true for a general probability
  vector $a^{S}_{0}\!$ as initial condition, with the marginalised
  initial condition $a^{U}_{0}\!$ according to 
  Eq.~\eqref{eq:def-marginal} on the subsystem.
\end{prop}

\begin{proof}
  Let $\{ a^{S}_{t} \mid t\geqslant 0\}$ be a solution of
  Eq.~\eqref{eq:gen-reco-coeff-DGL}, and let $\cA\in\ts\PP(U)$ be
  fixed. Then, we have
\[
\begin{split}
   \dot{a}^{\ts U}_{t} \nts (\cA) \; & = 
    \sum_{\substack{\cC\in\ts\PP(S) \\ \cC|^{\pa}_{U} = \cA}}
     \! \dot{a}^{S}_{t} (\cC)\, = \!
    \sum_{\substack{\cC\in\ts\PP(S) \\ \cC|^{\pa}_{U} = \cA}}
    \Bigl(-\rtot{} \, a^{S}_{t} (\cC) +
    \sum_{\udo{\cD}\succcurlyeq\cC} \gamma^{S} (a^{S}_{t}; \cC,\cD) \,
    \varrho^{S} (\cD) \Bigr) \\
  & = \, -\rtot{} \, a^{U}_{t} \nts\nts (\cA) \; +
    \sum_{\substack{\cC\in\ts\PP(S) \\ \cC|^{\pa}_{U} = \cA}}\,
    \sum_{\udo{\cD}\succcurlyeq\cC} \gamma^{S} (a^{S}_{t}; \cC,\cD) \,
    \varrho^{S} (\cD) \\
  & = \, \, -\rtot{} \, a^{U}_{t} \nts\nts (\cA) \; +
    \sum_{\substack{\cD\in\ts\PP(S) \\ \cD|^{\pa}_{U} 
    \succcurlyeq \cA}}\, \biggl(\,\sum_{\substack{\udo{\cC}
           \preccurlyeq\cD \\ \cC|^{\pa}_{U}=\cA }} 
     \! \! \gamma^{S} (a^{S}_{t}; \cC,\cD)\biggr)\,
    \varrho^{S} (\cD) \\
  & = \, -\rtot{} \, a^{U}_{t} \nts\nts (\cA) \; +
    \sum_{\udo{\cB} \succcurlyeq \cA}\,
    \sum_{\substack{\cD\in\ts\PP(S) \\ \cD|^{\pa}_{U} = \cB}} 
       \! \varrho^{S} (\cD)\!
    \sum_{\substack{\udo{\cC}\preccurlyeq\cD \\ \cC|^{\pa}_{U} = \cA}}
    \! \!  \gamma^{S} (a^{S}_{t}; \cC,\cD) \\
  & = \,  -\rtot{} \, a^{U}_{t} \nts\nts (\cA) \; +
    \sum_{\udo{\cB} \succcurlyeq \cA}\!
     \gamma^{U} \! (a^{U}_{t}\! ; \cA,\cB\ts )\!
    \sum_{\substack{\cD\in\ts\PP(S) \\ \cD|^{\pa}_{U} = \cB}} 
      \! \! \varrho^{S} (\cD) \\
  & = \, - \big (\rtot{} - \varrho(\pmax) \big )
      \, a^{U}_{t} \nts\nts (\cA) \; +
    \sum_{\cA \preccurlyeq \udo{ \cB} \prec \pmax}
    \varrho^{\ts U} \! (\cB\ts )
    \prod_{i=1}^{|\cB|} a^{B_i}_{t} (\cA|^{\pa}_{\nts B_i})\ts ,
\end{split}
\]
where we have used Lemma~\ref{lem:marginal} in the second-last step.
This made the $\gamma$-term independent of $\cD$, which in turn
allowed the last step on the basis of Eq.~\eqref{eq:def-induced-rates}
and another application of Lemma~\ref{lem:marginal}. The second-last
step shows that the marginalised family indeed satisfies the proper
ODE for the subsystem as defined for $\varnothing \ne U \subset S$
via Eq.~\eqref{eq:gen-reco-coeff-DGL}, as it must; the last step then
leads to Eq.~\eqref{eq:a-marg}.

It is an easy exercise that $\delta^{\ts U}\! (\cA,\pmax)$ is the
initial condition for the subsystem that emerges as the
marginalisation of the original initial condition $\delta^{S}
(\cB,\pmax)$. Since $\cA\in\ts\PP(U)$ was arbitrary, the main
statement is proved.  The last claim is an obvious generalisation.
\end{proof}

Note that Proposition~\ref{prop:marg-consistent} can also be viewed as
a consequence of Proposition~\ref{prop:measure-subsystem} and
Theorem~\ref{thm:equivalence}.  We have nevertheless opted for an
explicit verification because several steps of the above proof will
reappear when we proceed to a solution of the recombination equation.

\section{The backward point of view: Partitioning process}
\label{sec:partitioning}

Now that we have understood the structure of the ODE for $a^{}_t$ in
the usual (forward) direction of time, let us consider a related
(stochastic) process that will provide an additional meaning for
$a^{}_t$.  Let $\{\varSigma^S_t\}^{\pa}_{t\geqslant 0}$ be a Markov
chain in continuous time with values in $\PP(S)$ that is constructed
as follows. Start with $\varSigma^S_0= \pmax$. If the current state is
$\varSigma^S_t=\cC$, then part $C_i$ of $\cC$ is replaced by $\sigma
\in \PP(C_i)$ at rate $\varrho^{C_i} (\sigma)$, independently of all
other parts. That is, the transition from $\cC$ to $(\cC \setminus
C_i) \sqcup \sigma$ occurs at rate $\varrho^{C_i} (\sigma)$ for for
all $\pmax \ne \sigma \in \PP(C_i) $ and $1 \leqslant i \leqslant
|\cC|$.  Obviously, $\{\varSigma^S_t\}^{\pa}_{t\geqslant 0}$ is a
process of progressive refinements, which we call the
\emph{partitioning process}.  Likewise, we define the partitioning
process $\{\varSigma^U_t\}^{\pa}_{t\geqslant 0}$ on $\PP(U)$ for
$\varnothing \ne U\subset S$ in the same way as
$\{\varSigma^S_t\}^{\pa}_{t\geqslant 0}$, but based on the marginal
recombination rates $\varrho^U$.

Now let $P_{t}^{U} \! (\cC, \cD) := \mathbf{P}\bigl(\varSigma_t^U=\cD
\mid \varSigma_0^U = \cC\bigr)$, where $\mathbf{P}$ denotes
probability. That is, $P_{t}^{U} \! (\cC, \cD)$ is the transition
probability (in standard notation) from `state' $\cC$ to `state' $\cD$
during a time interval of length $t$.  Since
$\{\varSigma^U_t\}^{\pa}_{t\geqslant 0}$ is a process of progressive
refinements, it is obvious that
\begin{equation}\label{eq:refinement}
  P_{t}^{U}\! (\cC, \cD) \, = \, 0 \quad \text{ for }
   \cD \not\preccurlyeq \cC.
\end{equation}
Furthermore, since the parts are (conditionally) independent 
once they appear, we have
\begin{equation}\label{eq:cond-independence}
  P_{t}^{U} \! (\cC, \cD) \, = \, \prod_{i=1}^{|\cC|}  
  P^{C_i}_{t} \nts (\pmax, \cD|^{\pa}_{C_i}) 
  \quad \text{for }  \cD \preccurlyeq \cC .
\end{equation}

Let us now consider the distribution of $\varSigma_{t}^{U} \! $, that
is, the collection $\big\{ P_{t}^{U} \! (\pmax, \cA)\big\}_{\cA \in
  \PP(U)}$.  Clearly, the initial value is $P_{0}^{U} \! (\pmax,
\mathcal{A}) = \delta(\cA,\pmax)$.  The time evolution is given by
\begin{equation} 
\begin{split}\label{eq:KBE}
  \frac{\dd}{\dd t} P_t^{U}\! (\pmax, \cA)
   \, & = \,  -  \bigl(\rtot{} - \varrho^{U} \! (\pmax) \bigr) 
       P_{t}^{U} \! (\pmax, \cA) \, +
     \sum_{\udo \cB \prec \pmax} \varrho^{U}\! (\cB) 
        \ts \ts P_{t}^{U} \! (\cB, \cA) \\
     & =\, -  \bigl(\rtot{} - \varrho^{U} \! (\pmax) \bigr) 
        P_{t}^{U} \! (\pmax, \cA) \, + \!
  \sum_{\cA \preccurlyeq \udo \cB \prec \pmax} \! \varrho^{U} \! (\cB)
  \prod_{i=1}^{|\cB|}  P^{B_i}_{t} \nts (\pmax, \cA|_{B_i}).
\end{split}
\end{equation}
Here, the first step is an application of the \emph{Kolmogorov
  backward equation} (for background, see \cite[Ch.~XVII.8]{Feller} or
\cite[Ch.~2.1]{Norris}), namely, the decomposition according to the
first transition away from $\pmax$, which is to state $\cB$ with rate
$\varrho^{U}\! (\cB)$ for all $\cB \prec \pmax$. The second step uses
Eqs.~\eqref{eq:refinement} and \eqref{eq:cond-independence}.  The
argument is illustrated in Figure~\ref{fig:tree}.

\begin{figure}[t]

\begin{picture}(280,150)(70,110)

\put(200, 120){\makebox (0, 0){$\bullet$}}
\put(220, 120){\makebox (0, 0){$\pmax$}}

\multiput(200, 120)(,){1}{\line(1, 1){30}}
\multiput(200, 120)(,){1}{\line(-1, 2){15}}
\multiput(200, 120)(,){1}{\line(-1, 1){30}}
\multiput(200, 120)(,){1}{\line(-1, 4){8}}
\put(215, 150){\makebox (0, 0){$\cdots$}}

\put(170, 150){\makebox (0, 0){$\bullet$}}
\put(230, 150){\makebox (0, 0){$\bullet$}}
\put(185, 150){\makebox (0, 0){$\bullet$}}
\put(192, 150){\makebox (0, 0){$\bullet$}}

\put(205, 200){\makebox (0, 0){$\vdots$}}

\multiput(170, 150)(,){1}{\line(-1, 4){8}}
\multiput(170, 150)(,){1}{\line(-1, 1){30}}
\multiput(170, 150)(,){1}{\line(-2, 1){57}}

\multiput(185, 150)(,){1}{\line(0, 1){30}}
\multiput(192, 150)(,){1}{\line(0, 1){30}}
\multiput(230, 150)(,){1}{\line(1, 1){30}}

\put(140, 180){\makebox (0, 0){$\bullet$}}
\put(162, 180){\makebox (0, 0){$\bullet$}}
\put(185, 180){\makebox (0, 0){$\bullet$}}
\put(192, 180){\makebox (0, 0){$\bullet$}}
\put(260, 180){\makebox (0, 0){$\bullet$}}
\put(111, 180){\makebox (0, 0){$\bullet$}}

\put(225,180){\makebox (0, 0){$\cdots$}}

\put(100, 220){\makebox (0, 0){$\bullet$}}
\put(120, 220){\makebox (0, 0){$\bullet$}}
\put(130, 220){\makebox (0, 0){$\bullet$}}
\put(150, 220){\makebox (0, 0){$\bullet$}}
\put(160, 220){\makebox (0, 0){$\bullet$}}
\put(180, 220){\makebox (0, 0){$\bullet$}}
\put(190, 220){\makebox (0, 0){$\bullet$}}
\put(200, 220){\makebox (0, 0){$\bullet$}}
\put(220, 220){\makebox (0, 0){$\bullet$}}
\put(230, 220){\makebox (0, 0){$\bullet$}}
\put(250, 220){\makebox (0, 0){$\bullet$}}
\put(270, 220){\makebox (0, 0){$\bullet$}}
\put(280, 220){\makebox (0, 0){$\bullet$}}
\put(320, 120){\makebox (0, 0){$-$}}

\multiput(320, 120)(,){1}{\line(0, 1){30}}

\put(320, 150){\makebox (0, 0){$-$}}
\put(250, 150){\makebox (0, 0){$\mathcal{B}$}}
\multiput(320, 150)(,){1}{\line(0, 1){30}}
\put(320, 220){\makebox (0, 0){$-$}}
\put(320, 180){\makebox (0, 0){$-$}}
\put(280, 180){\makebox (0, 0){$\mathcal{C}$}}
\multiput(320, 180)(,){1}{\vector(0, 1){60}}

\put(300, 220){\makebox (0, 0){$\mathcal{A}$}}

\end{picture}
\caption{\label{fig:tree} A sketch of the partitioning process. Each
  dot represents one part of the corresponding partition, while the
  lines indicate the partitioning.}
\label{fig}
\end{figure}
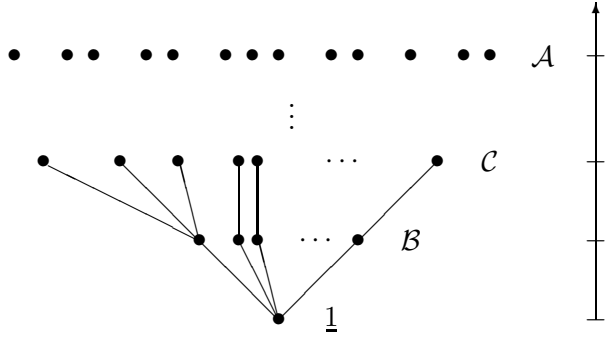

Comparing Eq.~\eqref{eq:KBE} with Eq.~\eqref{eq:a-marg}, we see that
the quantities $P^{U}_{t} \! (\pmax,\cA)$ and $a^{\ts U}_{t} \nts\nts
(\cA)$, with $U \subseteq S$ and $\cA \in \PP(U)$, satisfy the same
collection of ODEs, and the initial values also agree at $P_{0}^{U} \!
(\pmax, \cA) \nts =\delta_{\mathcal{A},\pmax}=a^{U}_{0}\!
(\mathcal{A})$.  By an obvious inductive argument, the two families
can thus be identified. We have therefore shown the following result.

\begin{theorem}\label{thm:a-part}
  The probability vector\/ $a_{t}^{U}\!$ agrees with the distribution
  of the partitioning process\/ $\varSigma^{U}_{t}\!$.  Explicitly, we
  have
\[
      a_{t}^{U} \! (\cA) \, = \, 
     P_{t}^{U} \! (\pmax, \mathcal{A}) \, = \, 
     \mathbf{P} \bigl(\varSigma_{t}^{U} = \cA \mid 
        \varSigma_{0}^{U} = \pmax \bigr)
\]
for all\/ $U \subseteq S$, $\cA \in \PP(U)$ and all\/ 
$t \geqslant 0$.   \qed
\end{theorem}

As the reader may have noticed, the partitioning process was defined
so as to reflect the action of recombination on the ancestry of the
genetic material of an individual \emph{backward in time.}  Namely, if
a sequence is pieced together according to a partition $\cA= \{A_1,
\ldots, A_r\}$ from various parents forwards in time, this implies
that the sequence is partitioned into the parts of $\cA$ when we look
backwards in time, where each part is associated with a different
parent. In this light, Theorem~\ref{thm:a-part} means the following:
If we follow the ancestry of the genetic material of an individual
from the present population (that is, starting at time $t$) backward
in time, then $a_t^S(\cA)$ is the probability that the sites are
partitioned into different parents according to $\cA$ at time $t$
before the present. That is, the sites in $A_1$ go back to one
individual in the initial population, the sites in $A_2$ go back to a
second individual and so on, and the sites in $A_r$ go back to an
$r$th individual.  The partitioning process is the deterministic limit
of the corresponding stochastic process in finite populations, namely,
the \emph{ancestral recombination graph} (ARG); see
\cite[Ch.~3.4]{Durrett}, but note that the ARG is usually described
for single crossovers only.

At the same time, the above suggests a nice interpretation of the
solution of the recombination equation in
Theorem~\ref{thm:equivalence}.  Indeed, the type distribution at
present may be obtained in a two-step procedure. In the first step,
one decides how the sites of a present individual have been pieced
together from different parents from the initial population.  In the
second step, one assigns letters to the various parts (that is,
parents): If the corresponding partition is $\cA=\{A_1, \ldots,
A_r\}$, then the letters for the sites in $A_1$ are drawn from
$\pi_{\! A_1}^{} \nts . \ts \omega^{}_0$, and so on, until the letters
for the sites in $A_r$ are drawn from $\pi_{\! A_r}^{} \nts . \ts
\omega^{}_0$ --- and this independently, due to the infinite
population.  The resulting type distribution of our individual today
is $R_{\cA}(\omega^{}_0)$.  What we have described here in words
amounts to a \emph{duality relation} between the recombination
dynamics forward in time and the partitioning process backward in
time. We do not go into detail here; for more, see \cite{BW} and
\cite{EPB}.

\section{Generic case:\ Recursive solution}
\label{sec:gen-sol}

Motivated by the M\"obius formula in Eq.~\eqref{eq:lin-decay} and by a
reminiscent structure in the solution of the discrete recombination
equation, compare \cite{EB-ICM}, we now restrict
Eq.~\eqref{eq:gen-reco-coeff-DGL} to the forward-invariant simplex
$\cP ( \PP (S))$, compare Proposition~\ref{prop:reduction-props}, and
make the \emph{ansatz}
\begin{equation}\label{eq:gen-ansatz}
    a^{U}_{t} \! (\cA) \, = \, \sum_{\udo{\cB}\succcurlyeq \cA} 
    \theta^{\ts U}\! (\cA,\cB\ts ) \, \ee^{- \psi^{\ts U}\! (\cB\ts )\ts  t}
\end{equation}
for any non-empty $U\nts \subseteq S$, with decay rates $ \psi^{\ts U}\!
$ and coefficient functions $\theta^{\ts U}\! $.  Our aim is to
determine $\theta^{\ts U}\! $ by a comparison of coefficients once we
know a (recursive) formula for the rate function $ \psi^{\ts U}\! $
together with linear independence of the exponentials in
Eq.~\eqref{eq:gen-ansatz}. Note that this will then also solve the ODE
\eqref{eq:gen-reco-coeff-DGL} with an initial condition from $\cM_{+}
( \PP (S))$ via multiplying the right-hand side of
Eq.~\eqref{eq:gen-ansatz} by the norm of the initial condition.

From our results of Section~\ref{sec:linear}, see 
Theorem~\ref{thm:better-than-nothing} and
Remark~\ref{rem:lin-decay-inverse} in comparison 
with Eq.~\eqref{eq:gen-ansatz}, we know that
\begin{equation}\label{eq:decay-max}
     \psi^{\ts U} \! (\pmax) \, = \, \chi^{\ts U} \!(\pmax)  \, =
     \sum_{\cA\ne \pmax} \varrho^{\ts U} \! (\cA) \, = \,
     \rtot{} - \varrho^{\ts U} \! (\pmax) \ts .
\end{equation}
Note that we have used the fact that $\mu^{U}\! (\pmax,\pmax) =
\theta^{U}\! (\pmax,\pmax)=1$ for all $\varnothing\ne U\subseteq S$,
which is obvious.
In particular, this gives $\psi^{\ts U} \! (\cA) = 0$ for any $U$ with
$|U|=1$, which is the trivial limiting case with an empty sum on the
right-hand side.  All other values are now defined recursively by
\begin{equation}\label{eq:decay-rec}
   \psi^{\ts U} \! (\cA) \, := \sum_{i=1}^{|\cA|} \psi^{A_i} (\pmax) \ts ,
\end{equation}  
which is motivated by the well-known eigenvalues in the two-parent,
discrete-time analogue of our model; see \cite[Theorem~6.4.3]{Lyu} or
\cite{Nag}.  In the context of our partitioning process, $\psi^{A_i}
(\pmax)$ is the total rate of any further partitioning of part $A_i$,
and so, due to the independence of the parts, $\psi^{\ts U} \! (\cA)$
is the total rate of transitions out of state $\cA$. For any $U$ with
$|U|=2$, one has $\psi^{\ts U} \! (\cA) = \delta^{U}\! (\cA,\pmax)$,
while one finds $\psi^{\ts U} \! (\pmin) = 0$ for all non-empty
$U\nts\subseteq S$. This is consistent with the fact that no further
partitioning (or `decay') is possible when starting from the partition
$\pmin$ of $U$.

The decay rates defined this way have important summation properties
as follows.
\begin{lemma}\label{lem:decay-sum}
  Let\/ $S$ be a finite set and\/ $\varnothing \ne U\nts\subseteq S$
  be arbitrary, but fixed.  If\/ $\cA = \{ A_1, \ldots, A_r \}$ is a
  partition of\/ $U$, with\/ $r = \lvert \cA \rvert$, one has\/ $
  \psi^{U}\! (\cA) + \sum_{i=1}^{r} \varrho^{A_{i}} (\pmax) = r \ts
  \rtot{}$. If\/ $s^{\pa}_{\! \cA}$ is the number of the parts of\/
  $\cA$ that are singletons, this can be simplified as
\[
    \psi^{U}\! (\cA) \, + \! \sum_{\substack{i=1_{\vph} \\ |A_{i}| > 1}}^{r}
    \! \! \varrho^{A_{i}} (\pmax) \, = \, \bigl( r - s^{\pa}_{\! \cA}
    \bigr) \rtot{} \ts .
\]  
In particular, $\psi^{U}\! (\pmin) = 0$ holds for all non-empty\/
$U\nts\subseteq S$.
 
Moreover, for arbitrary\/ $\cB,\cC\in\PP(U)$ with\/
$\cC\preccurlyeq\cB$, one has
\[
   \psi^{U} \! (\cC) \, = \sum_{i=1}^{|\cB|} \psi^{B_i}
    (\cC|^{\pa}_{\nts B_i} ) \ts .
\]
\end{lemma}
\begin{proof}
  The first claim follows from the relation $\psi^{A_{i}} \nts (\pmax)
  = \rtot{} - \varrho^{A_{i}}\nts (\pmax)$, where $\rtot{}$ does
  \emph{not} depend on $U$ as a consequence of
  Eq.~\eqref{eq:rtot}. For any part $A_{i}$ that is a singleton, one
  has $\varrho^{A_{i}}\nts (\pmax) = \rtot{}$, which implies the
  second identity. The latter, in turn, confirms that $\psi^{U}\!
  (\pmin) = 0$ for all $U\nts\subseteq S$, since $s^{\pa}_{\!\cA}=r$ in
  this case.

  For the remaining identity, fix $\cB = \{ B_{1}, \dots, B_{r} \} \in
  \PP(U)$. If $\cC = \{ C_{1}, \dots, C_{s} \} \in \PP(U)$ is a
  refinement of $\cB$, we have $s\geqslant r$ and $\cC$ must be a
  joining of the form
\[
    \cC \, = \, \cC_{1} \sqcup \dots \sqcup \cC_{r} 
\]
with $\cC_{i} = \cC|^{\pa}_{\nts B_i} \in \PP(B_{i})$ for all $1
\leqslant i\leqslant r$. Writing $\cC_{i} = \bigl\{ C^{(i)}_{1}, \dots
, C^{(i)}_{k_i} \bigr\}$, we have $s=\sum_{i=1}^{r} k_{i}$ because
$\cC\preccurlyeq \cB$ implies the existence of a bijection between the
two families $\{ C^{\pa}_{k} \mid 1 \leqslant k \leqslant s \}$ and
$\{ C^{(i)}_{j} \mid 1 \leqslant i \leqslant r \text{ and } 1
\leqslant j \leqslant k_{i} \}$.  Eq.~\eqref{eq:decay-rec} now implies
that the right-hand side of our claim is
\[
   \sum_{i=1}^{|\cB|} \psi^{B_i} (\cC_{i} ) \, =  \sum_{i=1}^{r}
   \sum_{j=1}^{k_{i}} \psi^{C^{(i)}_{\nts j}} (\pmax) \, =
   \sum_{\ell=1}^{s} \psi^{C_{\ell}} (\pmax) \, = \, 
   \psi^{\ts U} \! (\cC) \ts ,
\]
and the argument is complete.
\end{proof}

If $U\ne \varnothing$, with $|U| \leqslant 3$, one sees with
Corollary~\ref{coro:add-on} that
\begin{equation}\label{eq:psi-versus-chi}
  \psi^{U} \! (\cA) \, = \, \chi^{U} \! (\cA) \, = \!
  \sum_{\cB \notin [\cA,\pmax]}\! \varrho^{U} \! (\cB)
\end{equation}
holds for any $\cA \in \PP (U)$. For larger $U$, this relation holds
for $\cA = \pmax$ and any $\cA\in\PP (U)$ with $|\cA|=2$ if one of the
parts is a singleton set, as follows from
Theorem~\ref{thm:better-than-nothing}.

\begin{lemma}\label{lem:distinct}
  Let the decay rates\/ $\psi^{S} \nts (\cA)$ with\/ $\cA\in\PP (S)$
  be distinct and assume that the recombination rates for subsystems
  are calculated according to Eq.~\eqref{eq:def-induced-rates}. Then,
  for any non-empty\/ $U\nts\subseteq S$, the decay rates\/ $\psi^{U}\!
  (\cB)$ with\/ $\cB\in\PP (U)$ are distinct as well. In this case,
  the exponential functions\/ $\ee^{-\psi^{U}\! (\cB)\ts t}$ with\/
  $\cB\in\PP(U)$ are linearly independent over\/ $\RR$.
\end{lemma}

\begin{proof}
  Assume distinctness of the rates $\psi^{S} (\cA)$ with $\cA\in\PP
  (S)$ and fix a non-empty $U\nts\subseteq S$. If $\cB, \cB' \in \PP
  (U)$ exist with $\cB\ne \cB'$ but $\psi^{U}\! (\cB) = \psi^{U} \!
  (\cB')$, we may choose $\cA, \cA' \in \PP (S)$ with $\cA|^{\pa}_{U}
  = \cB$, $\cA'|^{\pa}_{U} = \cB'$ (hence $\cA\ne \cA'$) and
  $\cA|^{\pa}_{U_{\vph}^{\mathsf{c}}} =
  \cA'|^{\pa}_{U_{\vph}^{\mathsf{c}}}$.  Since $S = U \,\dot{\cup}\,
  U_{}^{\mathsf{c}}$, we can employ Lemma~\ref{lem:decay-sum} to
  derive that
\[
   \psi^{S} (\cA) \, = \, \psi^{U}\! \bigl(\cA|^{\pa}_{U} \bigr) 
   + \psi^{U^{\mathsf{c}}}\! \bigl(\cA|^{\pa}_{U_{\vph}^{\mathsf{c}}} 
   \bigr) \, = \, \psi^{U}\! \bigl(\cA'|^{\pa}_{U} \bigr) 
   + \psi^{U^{\mathsf{c}}} \! \bigl(\cA'|^{\pa}_{U_{\vph}^{\mathsf{c}}} 
   \bigr) \, = \, \psi^{S} (\cA') \ts ,
\]
which is a contradiction to the distinctness assumption. This proves
the first claim, while the second assertion is standard.
\end{proof}

Below, we need to understand the decay rates in various ways.  Let us
thus expand on their relation with the recombination rates.  Starting
from the definition in Eq.~\eqref{eq:decay-rec}, with $\cA = \bigl\{
A^{\pa}_{1}, \ldots , A^{\pa}_{\lvert \cA \rvert} \bigr\}$, one
obtains
\begin{align}
   \nonumber
   \psi^{U}\! (\cA) \, & = \sum_{i=1}^{\lvert \cA\rvert}
   \bigl( \rtot{} - \varrho^{A_i} \nts (\pmax) \bigr)
   \, = \, \lvert\cA\rvert \, \rtot{} \, - 
   \sum_{i=1}^{\lvert \cA\rvert}\, \sum_{\cB\in\PP (U)} \!
   \mathbf{1}^{\pa}_{\{\cB | ^{\pa}_{\nts A_i}\nts = \pmax\}} \, 
   \varrho^{U}\! (\cB) \\
   & = \, \lvert\cA\rvert \, \rtot{} \, - 
    \sum_{\cB\in\PP (U)} \biggl(\, \sum_{i=1}^{\lvert \cA\rvert}\,
     \mathbf{1}^{\pa}_{\{\cB | ^{\pa}_{\nts A_i}\nts = \pmax\}}\nts \biggr)
     \varrho^{U}\! (\cB) \, = \! \sum_{\cB\in\PP (U)}
     \biggl( \lvert \cA \rvert \, -  \sum_{i=1}^{\lvert \cA\rvert}\,
     \mathbf{1}^{\pa}_{\{\cB | ^{\pa}_{\nts A_i}\nts = \pmax\}}\nts \biggr)
     \varrho^{U}\! (\cB) \ts ,
     \label{eq:psi-rates}
\end{align}
where $\bf{1}^{\pa}_{\{ \dots\}}$ denotes the characteristic function
of the condition in its subscript.

\begin{lemma}\label{lem:decay-reco-rates}
   The decay rates defined in Eq.~\eqref{eq:decay-rec} are linear
   functions of the recombination rates of the form\/
   $\psi^{U}\! (\cA) = \sum_{\cB\in\PP (U)} \kappa^{U}\! (\cA,\cB)
   \ts \varrho^{U}\! (\cB)$, with
\[
    \kappa^{U}\! (\cA,\cB) \, = \, \lvert \cA \rvert \, -  
     \sum_{i=1}^{\lvert \cA\rvert}
     \mathbf{1}^{\pa}_{\{\cB | ^{\pa}_{\nts A_i}\nts = \pmax\}} \ts .
\]   
  In particular, the coefficients\/ $\kappa^{U}\! (\cA,\cB)$ are
  non-negative integers, and one has $\kappa^{U}\! (\cA,\cB) = 0$
  if and only if\/ $\cB \succcurlyeq\cA$.
\end{lemma}
\begin{proof}
  The first claim summarises our calculation in
  Eq.~\eqref{eq:psi-rates}. It is obvious from the explicit formula
  that $\kappa^{U}\! (\cA,\cB) \in \NN^{\pa}_{0}$, for all $\cA, \cB
  \in \PP (U)$. The coefficient vanishes if and only if each condition
  under the summation is true, which means $\cB | ^{\pa}_{\nts A_{i}}
  = \pmax = \{ A_{i} \}$ for all $1 \le i \le \lvert \cA \rvert$. But
  the latter condition, in turn, is equivalent with $\cB
  \succcurlyeq\cA$, which proves the assertion.
\end{proof}

\begin{remark}\label{rem:generic}
  One important consequence of Lemma~\ref{lem:decay-reco-rates} is
  that, as linear functions of the recombination rates, $\psi^{U}\!
  (\cA)$ and $\psi^{U}\! (\cB)$ are different whenever $\cA \ne \cB$,
  because the intervals $[\cA,\pmax]$ and $[\cB,\pmax]$ are then
  different, too.  By a standard Baire category argument, the
  situation that the rates $\psi^{U}\! (\cA)$ with $\cA\in\PP (U)$ are
  distinct is then generic. In fact, the exceptional set with any
  degeneracy among the decay rates is a union of true hyperplanes in
  parameter space $\{\varrho^{U}\! (\cA) \ge 0 \mid \cA\in\PP (U) \}
  \simeq \RR_{+}^{B(\lvert U\rvert)}$.  Consequently, it is both a
  nowhere dense set and a null set. The corresponding statement
  remains true if we restrict our attention to the parameter space $\{
  \varrho^{U}\!  (\cA) \ge 0 \mid \cA\in\PP_{2} (U) \} \simeq
  \RR_{+}^{2^{|U| - 1} - 1}$.  \exend
\end{remark}

Let us now return to the original ansatz \eqref{eq:gen-ansatz}.  The
coefficient function $\theta^{\ts U}$ is an element of the incidence
algebra $\AAA(\PP(U))$. Since $a^{U}_{0} (\cA) = \delta^{\ts U} \!
(\cA,\pmax)$ is our initial condition, we see that, for $\varnothing
\ne U\nts \subseteq S$ and any $\cA\in\PP(U)$, we must have
\[
     \sum_{\udo{\cB} \succcurlyeq \cA} \theta^{\ts U} \! ( \cA,\cB\ts )
     \, = \, \delta^{\ts U} \! (\cA,\pmax ) \ts ,
\]
hence $\theta^{\ts U} \! (\pmax,\pmax) = 1$ (as mentioned above) 
together with
\begin{equation}\label{eq:theta-initial}
    \theta^{\ts U} \! (\cA,\pmax) \, = \, - \! \!
    \sum_{\cA \preccurlyeq \udo{\cC}\prec \pmax}
    \! \theta^{\ts U} \! (\cA,\cC) 
\end{equation}
for $\cA \prec \pmax$.  This way, for non-empty $U\nts\subseteq S$ and
$\cA\in\PP(U)$, the coefficients $\theta^{\ts U}\! (\cA,\pmax)$ are
either fixed or recursively determined from $\theta^{\ts U}\!
(\cA,\cC)$ with $\cA\preccurlyeq\cC\prec \pmax$.

Taking the derivative of our ansatz \eqref{eq:gen-ansatz} leads to
\begin{equation}\label{eq:a-derivative}
    \dot{a}^{\ts U}_{t} \! (\cA) \, = \, - \sum_{\udo{\cB}\succcurlyeq\cA}
    \psi^{\ts U} \! (\cB\ts ) \, \theta^{\ts U} \! (\cA,\cB\ts ) \, 
    \ee^{-\psi^{\ts U} \! (\cB\ts )\ts t} ,
\end{equation}
which has to be compared with the right-hand side of
Eq.~\eqref{eq:a-marg}.  Inserting the corresponding expression for the
coefficients of the subsystem according to our ansatz yields the
expression
\[
\begin{split}
    \dot{a}^{\ts U}_{t} \! (\cA) \, & = \, - \rtot{} \, 
     a^{U}_{t} \! (\cA) \, +
    \sum_{\udo{\cB} \succcurlyeq \cA} \varrho^{\ts U} \! (\cB\ts ) 
     \, \prod_{i=1}^{|\cB|} \,
    \sum_{\substack{\cC_{i}\in\PP(B_i) \\ \cC_{i} \succcurlyeq 
       \cA|^{\pa}_{\nts B_i}}} \! \theta^{B_i} 
       (\cA|^{\pa}_{B_i},\cC_{i}) \,
    \ee^{-\psi^{B_i} (\cC_{i}) \ts t} \\
    & = \, - \rtot{} \, a^{U}_{t} \! (\cA) \, + 
    \sum_{\udo{\cB} \succcurlyeq \cA} 
    \varrho^{\ts U} \! (\cB\ts )  \! \sum_{\cC\in [\cA,\cB]}
    \exp \Bigl( - \sum_{i=1}^{|\cB|} \psi^{B_i} 
       (\cC|^{\pa}_{\nts B_i}) \ts t \Bigr)
    \prod_{j=1}^{|\cB|} \theta^{B_j} (\cA|^{\pa}_{\nts B_j} , 
       \cC|^{\pa}_{\nts B_j}) \\[2mm]
    & = \, - \rtot{} \, a^{U}_{t} \! (\cA) \, + 
    \sum_{\udo{\cB} \succcurlyeq \cA} 
    \varrho^{\ts U} \! (\cB\ts )  \! \sum_{\cC\in [\cA,\cB]}
     \!\ee^{- \psi^{\ts U} \! (\cC) \ts t}\,
    \prod_{j=1}^{|\cB|} \theta^{B_j} (\cA|^{\pa}_{\nts B_j} , 
       \cC|^{\pa}_{\nts B_j}) \ts ,
\end{split}
\]
where the second part of Lemma~\ref{lem:decay-sum} was used 
in the last step. Inserting the expression \eqref{eq:gen-ansatz}
for $a^{U}_{t} \! (\cA)$, changing
summation variables and regrouping terms gives
\begin{eqnarray}
    \dot{a}^{\ts U}_{t} \! (\cA)  & = & - 
    \sum_{\udo{\cB}\succcurlyeq\cA}
    \ee^{-\psi^{\ts U} \! (\cB\ts )\ts t} \Bigl( \rtot{} \, 
     \theta^{\ts U} \! (\cA,\cB\ts ) \, -
    \sum_{\udo{\cC}\succcurlyeq\cB} \varrho^{\ts U} \! (\cC)
    \prod_{i=1}^{|\cC|} \theta^{C_i} (\cA|^{\pa}_{C_i} , 
       \cB|^{\pa}_{C_i})\Bigr) \nonumber \\
    & = & - \sum_{\udo{\cB}\succcurlyeq\cA}
    \ee^{-\psi^{\ts U} \! (\cB\ts )\ts t} \Bigl( \psi^{\ts U}\! (\pmax) \, 
       \theta^{\ts U} \! (\cA,\cB\ts ) \, -
    \! \sum_{\cB\preccurlyeq\udo{\cC}\prec \pmax}
       \! \varrho^{\ts U} \! (\cC)
    \prod_{i=1}^{|\cC|} \theta^{C_i} 
     (\cA|^{\pa}_{C_i} , \cB|^{\pa}_{C_i})\Bigr) ,
    \label{eq:rec-reduce}
\end{eqnarray}
where the second step effectively removes all terms with
$\varrho^{\ts U}\! (\pmax)$; compare Remark~\ref{rem:balance}.

Let $U\nts\subseteq S$ be non-empty and let us assume that the
exponential functions on the right-hand side of
Eq.~\eqref{eq:gen-ansatz} are linearly independent. By
Lemma~\ref{lem:distinct}, this is certainly the case when the rates
$\psi^{S} (\cA)$ with $\cA \in\PP(S)$ are distinct; compare
Remark~\ref{rem:generic}. Now, a comparison of coefficients in
Eq.~\eqref{eq:rec-reduce} with those of the right-hand side of
Eq.~\eqref{eq:a-derivative} yields the relations
\[
   \theta^{\ts U} \! (\cA,\cB\ts ) \bigl( \psi^{\ts U}\! (\pmax) - 
    \psi^{\ts U} \! (\cB\ts )\bigr)
   \, = \! \sum_{\cB\preccurlyeq\udo{\cC}\prec \pmax} \! 
        \varrho^{\ts U} \! (\cC) \prod_{i=1}^{|\cC|} 
       \theta^{C_i} (\cA|^{\pa}_{C_i} , \cB|^{\pa}_{C_i}) \ts ,
\]
for all $\cA,\cB\in\PP(U)$ with $\cA\preccurlyeq\cB$. Note that our
assumption entails the condition that $ \psi^{\ts U}\! (\pmax) \ne
\psi^{\ts U}\! (\cB)$ for all $\pmax \ne \cB\in\PP(U)$, so that we
get
\begin{equation}\label{eq:theta-rec}
   \theta^{\ts U} \! (\cA,\cB\ts ) \, = \! 
   \sum_{\cB\preccurlyeq\udo{\cC}\prec\pmax}
   \frac{\varrho^{\ts U} \! (\cC)}{\psi^{\ts U}\! 
    (\pmax) - \psi^{\ts U} \! (\cB\ts )}
   \,\prod_{i=1}^{|\cC|} \theta^{C_i} (\cA|^{\pa}_{C_i} , \cB|^{\pa}_{C_i})
\end{equation}
for all $\cA\preccurlyeq \cB \prec \pmax$, where the
$\theta$-coefficients for the subsystems are themselves determined
recursively, under the distinctness condition for the subsystem decay
rates, which follows from our assumption by
Lemma~\ref{lem:distinct}. Since we know $\theta^{U}\!$ for any $U$
with $1\leqslant |U| \leqslant 3$ from Corollary~\ref{coro:add-on} and
Remark~\ref{rem:lin-decay-inverse}, see also
Corollary~\ref{coro:simple-cases} below, the recursion
\eqref{eq:theta-rec} uniquely determines all $\theta$-coefficients.

We have thus shown the following result for the generic situation
of Remark~\ref{rem:generic}.
\begin{theorem}\label{thm:rec-sol}
  Let\/ $S\ne \varnothing$ be a finite set, and assume that the decay
  rates\/ $\psi^{S} (\cA)$ with\/ $\cA\in\PP(S)$ are distinct.  Let
  the decay rates for non-empty subsets\/ $U\nts\subseteq S$ be
  defined via Eqs.~\eqref{eq:decay-max} and \eqref{eq:decay-rec},
  where the marginal recombination rates on subsystems are given by
  Eq.~\eqref{eq:def-induced-rates}.

  Then, our exponential ansatz \eqref{eq:gen-ansatz} solves the Cauchy
  problem of Eq.~\eqref{eq:gen-reco-coeff-DGL} with initial
  condition\/ $a^{S}_{0} (\cA) = \delta^{S} (\cA,\pmax)$ if and only
  if the coefficients\/ $\theta^{S}\nts $ are recursively determined
  by Eq.~\eqref{eq:theta-rec} together with\/ $\theta^{U}\!
  (\pmax,\pmax) =1$ and Eq.~\eqref{eq:theta-initial}.  \qed
\end{theorem}

Let us mention in passing that a similar result can be derived for
other initial conditions as well. We concentrate on this one as it
fits Theorem~\ref{thm:equivalence} and thus also provides a solution
of the original recombination equation.

\begin{remark}
  The recursion in Eq.~\eqref{eq:theta-rec}, and the solution thus
  constructed for the recombination dynamics, is similar in structure
  to the corresponding relations (6.5.1) and (6.5.2) in Lyubich's book
  \cite{Lyu} for the two-parent case in discrete time. However, in
  contrast to our $\theta$'s, Lyubich's coefficients depend on
  $\omega^{}_{0}$.  As noted in \cite{Lyu}, this dependence on the
  initial condition poses a serious difficulty to the solution, since
  it makes an explicit iteration impossible.  Put differently, our
  approach achieves the complete separation of the recombination
  structure from the types, as already inherent in
  Theorem~\ref{thm:rec-sol} and further discussed in
  Section~\ref{sec:partitioning}; this entails a crucial
  simplification.  \exend
\end{remark}

In the situation of Theorem~\ref{thm:rec-sol}, we thus know that our
original ansatz \eqref{eq:gen-ansatz} leads to a solution.  Since we
then also know that $\psi^{S} (\cA) = 0$ holds only for $\cA = \pmin$,
and $\psi^{S} \nts (\cA) > 0$ otherwise, we have the following
immediate consequence for the asymptotic behaviour of the solution.

\begin{coro}
  Under the assumptions of Theorem~$\ref{thm:rec-sol}$, one has
\[
    \lim_{t\to\infty} a^{S}_{t} (\cA) \, = \, 
    \delta^{S} (\pmin,\cA) \ts ,
\]
which means that the corresponding solution\/ $\omega^{\pa}_{t}$ of
the ODE \eqref{eq:reco-eq}, with initial condition\/
$\omega^{\pa}_{0}$, $\|.\|$-converges to the equilibrium\/
$R^{\pa}_{\pmin} (\omega^{\pa}_{0})$ as\/ $t\to\infty$.  Moreover, the
convergence is exponentially fast, with the rate given by\/ $\min \{
\psi^{S} \nts (\cA) \mid \cA \ne \pmin \}$, while the individual rate
for the convergence of\/ $a^{S}_{t} (\cA)$ is given by\/ $\min \{
\psi^{S} \nts (\cB) \mid \pmin \ne \cB \succcurlyeq \cA \}$.  \qed
\end{coro}

The recursive nature of our solution does not immediately help 
to understand its structure and meaning. Let us thus study some 
aspects of it in more detail.

\section{Generic solution:\ Structure and further properties}
\label{sec:properties}

To begin, let us observe that Theorem~\ref{thm:rec-sol} has an
interesting consequence which contrasts the result of
Lemma~\ref{lem:decay-reco-rates} and shows a surprising structural
similarity with the corresponding relation of
Remark~\ref{rem:lin-decay-inverse}.
\begin{coro}\label{coro:nonlin-reco-decay}
  Let\/ $S$ be as in Theorem~$\ref{thm:rec-sol}$ and assume that the
  decay rates\/ $\psi^{S} (\cA)$ with\/ $\cA\in\PP (S)$ are distinct.
  Then, the recombination and the decay rates are related by
\[
      \varrho^{S} (\cA) \, = \, \rtot{} \, \delta^{S} (\cA,\pmax) \, -
      \sum_{\udo{\cB} \succcurlyeq \cA} \theta^{S} (\cA,\cB) \,
      \psi^{S} (\cB) ,
\]     
     and the corresponding relation holds for any subsystem
     that is defined by a non-empty\/ $U\nts \subseteq S$.
\end{coro}

\begin{proof}
  By Theorem~\ref{thm:rec-sol} and
  Proposition~\ref{prop:reduction-props}, we know that the ansatz of
  Eq.~\eqref{eq:gen-ansatz}, for $U=S$, leads to the unique solution
  of the ODE \eqref{eq:gen-reco-coeff-DGL} with initial condition
  $a^{S}_{0} (\cA) = \delta^{S} (\cA,\pmax) $. Thus, we may equate the
  right-hand side of Eq.~\eqref{eq:gen-reco-coeff-DGL} with that of
  Eq.~\eqref{eq:a-derivative}, again for $U=S$, and consider the
  resulting identity at $t=0$. Observing that
\[
    \gamma (a^{S}_{0} ; \cA, \cB) \, = \, \delta^{S} (\cA,\cB)
\]
holds for our probability vector $a^{S}_{0}$ as a result of
Eq.~\eqref{eq:def-gamma}, the first claim is clear.

The second assertion follows by the corresponding calculation with
$a^{U}_{t}\! $, which is justified by
Proposition~\ref{prop:marg-consistent} together with
Lemma~\ref{lem:distinct}.
\end{proof}

Note that the relation of Corollary~\ref{coro:nonlin-reco-decay} is a
nonlinear one, because the $\theta$-coefficients themselves generally
depend on the recombination rates. To make any further progress, we
need to better understand these coefficients, beyond $\theta^{U}\!
(\pmax,\pmax) = 1$ and the relations in Eq.~\eqref{eq:theta-initial}.

\begin{prop}\label{prop:invertible}
  Under the assumptions of Theorem~$\ref{thm:rec-sol}$, one finds the
  following properties.
\begin{enumerate}\itemsep=3pt
\item $\theta^{U} \! (\pmin,\pmin) = 1\vph$ for all non-empty\/
   $U\nts\subseteq S$;
\item $\sum_{\udo{\cA\ts}\nts \preccurlyeq\cB}\, \theta^{U} \! (\cA,\cB)
   = \delta^{U}\! (\pmin,\cB)$ holds for all\/ $\cB\in \PP (U)$;
 \item Assume further that\/ $\varrho^{S}$ is strictly positive on all
   partitions of\/ $S$ with two parts. Then, $\theta^{\ts U}\!
   (\cA,\cA) \ne 0$ for all\/ $\cA\in\PP (U)$ and all non-empty\/
   $U\nts\subseteq S$; In particular, $\theta^{\ts U}\!$ is then an
   invertible element of the incidence algebra for\/ $\PP (U)\nts$
   over\/ $\RR$.
\end{enumerate}
\end{prop} 

\begin{proof}
For the first claim, observe that $\pmin=\pmax$ for any $U$ with
$|U|=1$, so $\theta^{U} \! (\pmin,\pmin) = \theta^{U} \! (\pmax,\pmax)
= 1$ in this case. Assume now that the claim is true for all $U$ with
$\lvert U \rvert \leqslant r$, and consider a larger set, $U = \{
u^{\pa}_{1}, \dots , u^{\pa}_{r+1} \}$ say. With $\psi^{U} \!
(\pmin)=0$, we then get from Eq.~\eqref{eq:theta-rec} that
\[
  \theta^{U} \! (\pmin,\pmin) \, = \, \frac{1}{\psi^{U}\! (\pmax)}
  \sum_{\cC\ne\pmax}\varrho^{U} \! (\cC) \prod_{i=1}^{\lvert \cC \rvert}
  \theta^{C_{i}} (\pmin,\pmin) \, =  \frac{1}{\psi^{U}\! (\pmax)}
  \sum_{\cC\ne\pmax}\varrho^{U}\! (\cC) \, = \, 1 \ts ,   
\]
where the second step uses the induction hypothesis (note that we
always have $\lvert C_{i}\rvert \leqslant r$), while the last step
employs Eq.~\eqref{eq:psi-versus-chi} which applies here.

Next, the normalisation property of the $a^{U}_{t} \! (\cA)$ implies
\[
   1 \, = \! \sum_{\cA\in\PP (U)}\! a^{U}_{t} \! (\cA) \, = \!
   \sum_{\cA\in\PP (U)} \, \sum_{\udo{\cB}\succcurlyeq\cA}
   \theta^{U} \! (\cA,\cB) \, \ee^{-\psi^{U} \! (\cB) \ts t} \, =
   \! \sum_{\cB\in\PP (U)} \! \ee^{-\psi^{U} \! (\cB) \ts t}
   \sum_{\udo{\cA\ts}\nts\preccurlyeq\cB} \theta^{U} \! (\cA,\cB) \ts .
\]
Since $\psi^{U} \! (\pmin) = 0$ and since the decay rates $\psi^{U} \!
(\cB)$ with $\cB\in\PP (U)$ are distinct by our assumption together
with Lemma~\ref{lem:distinct}, the functions $\ee^{-\psi^{U}\!
  (\cB)\ts t}$ are linearly independent. Consequently, the last
identity is equivalent to the second claim.

For the third claim, recall that the dimension of the ODE system
\eqref{eq:gen-reco-coeff-DGL} is $B(n)$ if the set of partitions
$\cA\in\PP(S)$ with $\varrho^{S} (\cA) > 0$ generates the entire
lattice $\PP (S)$, which is the case under our assumptions. The convex
set that is spanned by the solution functions $a^{S}_{t}$ is then a
simplex of dimension $B(n) - 1$, and the normalisation condition
$\sum_{\cA\in\PP (S)}a^{S}_{t} (\cA) = 1$ is the only linear relation
between these functions.

We already know that $\theta^{S} (\pmax,\pmax) = 1$, and one easily
finds
\[
      \theta^{S} (\cA,\cA) \, = \, 
      \frac{\varrho^{S} (\cA)}{\psi^{S} (\pmax) - \psi^{S} (\cA)}
\]
for any $\cA\in\PP (S)$ with $\lvert \cA \rvert = 2$. One may now
proceed inductively in the number of parts. If there were some
$\cA\ne\pmin$ with $\theta^{S} (\cA,\cA) = 0$, where we may assume
$\cA$ to be the coarsest partition with this property, we would get
from Eq.~\eqref{eq:gen-ansatz} an additional linear relations among
the solution functions $a^{S}_{t}$, which is impossible. Since
$\theta^{S} (\pmin,\pmin) = 1$ by the second assertion, the claim is
true on the top level (defined by $S$).  Repeating the argument
for any non-empty $U\nts \subset S$ completes the argument.

Finally, the invertibility of $\theta^{U}\!$ as an element of the
incidence algebra is a standard consequence of
claim (3); compare \cite{Aigner}.
\end{proof}

Consequently, under the assumptions of part (3) of
Proposition~\ref{prop:invertible}, $\theta^{U}\!$ has a unique (left
and right) inverse, $\eta^{\ts U}$ say. This means that
\[
     \sum_{\cB \in [\cA,\cC]}\! \theta^{U} \! 
     (\cA, \cB) \, \eta^{U} \! (\cB,\cC)
     \, = \, \delta^{\ts U} \! (\cA,\cC) \, = \! 
     \sum_{\cB\in [\cA,\cC]}\!
     \eta^{U} \! (\cA,\cB) \, \theta^{U} \! (\cB,\cC)
\]
holds for all $\cA\preccurlyeq\cC$. The coefficients $\eta^{U}\!$ are
thus determined by $\eta^{U} \! (\cA,\cA) = 1/\theta^{U} \! (\cA,\cA)$
together with the recursion
\[
   \eta^{U} \! (\cA,\cC) \, = \, -\ts \frac{1}{\theta^{U} \! (\cA,\cA)} 
   \sum_{\cA \prec \udo{\cB} \preccurlyeq \cC} \! \theta^{U} \! (\cA,\cB) \,
   \eta^{U} \! (\cB,\cC)
\]
for $\cA\prec\cC$, which derives from the left equality above.
Alternatively, one may use the corresponding formula that derives from
the other identity.  

Either from direct calculations, or by invoking the results from
Section~\ref{sec:linear}, in particular
Theorem~\ref{thm:better-than-nothing} and Corollary~\ref{coro:add-on},
the following result is obvious (see the text after
Remark~\ref{rem:prob-vec} for the definitions of $\zeta$ and $\mu$).

\begin{coro}\label{coro:simple-cases}
  For any\/ $U\!$ with\/ $1\leqslant |U| \leqslant 3$, one has\/
  $\eta^{\ts U} = \zeta^{\ts U}\!$, and hence also\/ $\theta^{\ts U}\!
  = \mu^{U}\! $.  For larger sets\/ $U\!$, one has\/ $\eta^{\ts U}\!
  (\cA,\cA) = 1$ for\/ $\cA=\pmax$ and for all\/ $\cA\in\PP(U)$ with
  two parts, if one of them is a singleton set. In the latter case,
  also\/ $\theta^{\ts U}\!  (\cA,\cA) = 1$.
\end{coro}
\begin{proof}
  Corollary~\ref{coro:add-on} implies the first claim, while
  Theorem~\ref{thm:better-than-nothing} gives the second one for
  $\cA=\pmax$ as well as for any $\cA$ with two parts, provided one of
  them is a singleton set. The last claim follows from $\theta^{U}\!
  (\cA,\cA) \, \eta^{U}\! (\cA,\cA) = 1$.
\end{proof}

Let us now define
\begin{equation}\label{eq:def-b-coeff}
    b^{U}_{t} \! (\cA) \, = \sum_{\udo{\cB}\succcurlyeq \cA}
    \eta^{\ts U} \! (\cA,\cB\ts ) \, a^{U}_{t}\! (\cB\ts ) \ts ,
\end{equation}
which is an analogue of the summatory function from
Section~\ref{sec:linear} (see the proof of
Proposition~\ref{prop:lin-solve}).  Using $\eta^{U} \!\nts\nts *
\theta^{U} \!  = \delta^{U}$, we now obtain
\begin{equation}\label{eq:gen-summatory}
   b^{U}_{t}\! (\cA) \, = \, \ee^{-\psi^{\ts U} \! (\cA) \ts t} ,
\end{equation}
hence $b^{U}_{t}\! (\pmin) \equiv 1$ for all $\varnothing \ne
U\nts\subseteq S$ as a consequence of $\psi^{U}\! (\pmin) = 0$; compare
Lemma~\ref{lem:decay-sum}.  One difference to the linear case is that
the summation weights in Eq.~\eqref{eq:def-b-coeff} generally depend
on the recombination rates, while they were constant (in fact, given
by the $\zeta$-function of the incidence algebra) in
Section~\ref{sec:linear}. Due to the properties of the decay rates
$\psi^{U}\!$, one inherits the corresponding relations among the
coefficient functions $b^{U}_{t}$ for $\varnothing\ne U\nts\subseteq
S$. In particular,
\[
   b^{U}_{t}\! (\pmax) \, = \, \ee^{-\chi^{U}\! (\pmax)\ts t}
   \, =  \! \prod_{\pmax \ne \udo{\cA\ts}\nts \in \PP (U)}
   \! \! \ee^{-\varrho^{U}\! (\cA)\ts t}
\]
together with
\[
   b^{U}_{t}\! (\cA) \, = \, \prod_{i=1}^{|\cA|} b^{A_i}_{t} (\pmax)
\]
determines all coefficients, while Lemma~\ref{lem:decay-sum}
implies the additional relation
\[
   b^{U}_{t}\! (\cA) \, = \, \prod_{i=1}^{|\cB|} b^{B_i}_{t}
    (\cA|^{\pa}_{\nts B_i})
\]
for any $\cA,\cB\in\PP(U)$ with $\cA\preccurlyeq\cB$.

Note that the function $b^{U}_{t}\!$ generally does \emph{not} emerge
from $b^{S}_{t}$ via marginalisation in the sense of
Eq.~\eqref{eq:def-induced-rates}, which is another important
difference to the special situation of Section~\ref{sec:linear} and
Remark~\ref{rem:marg-sum}.

\begin{lemma}\label{lem:eta-props}
  Let\/ $\theta^{S}$ denote the coefficients from
  Theorem~$\ref{thm:rec-sol}$ in the case of distinct rates, and
  assume that\/ $\varrho^{S} (\cA) >0 $ for all\/ $\cA \in \PP (S)$
  with two parts, so that\/ $\theta^{S}$ is an invertible element of
  the incidence algebra. Then, for any non-empty\/ $U\nts \subseteq
  S$, one has the following properties.
\begin{enumerate}\itemsep=3pt
\item $\eta^{U} \! (\pmin,\pmin) = 1$;
\item $\eta^{U} \! (\pmin,\cA) = 1\vph$ holds for all\/ 
    $\cA\in\PP (U)$;
\item $\eta^{U} \! (\cA,\pmax) = 1\vph$ holds for all\/ $\cA\in\PP
   (U)$ and all\/ $\varnothing \ne U\nts\subseteq S$.
\end{enumerate}
\end{lemma}

\begin{proof}
  The invertibility of $\theta^{U}$ is clear from part (3) of
  Proposition~\ref{prop:invertible}, while $\eta^{U} \!
  (\pmin,\pmin) = 1$ is a consequence of the first assertion
  of the same proposition.

Next, we already know that $\eta^{U} \! (\pmin,\pmin) = 1$.  For $\cA
\succ \pmin$, we can proceed inductively via the standard inversion
formula \cite{Aigner} for $\eta^{U} \! = \bigl(\theta^{U}\bigr)^{\nts
  -1}$, which gives
\[
\begin{split}
   \eta^{U} \! (\pmin,\cA) \, & = \, -\ts \eta^{U} \! (\cA,\cA)
   \sum_{\udo{\cC}\prec \cA} \eta^{U} \! (\pmin,\cC) \,
   \theta^{U} \! (\cC,\cA) \, = -\ts \eta^{U} \! (\cA,\cA)
   \sum_{\udo{\cC}\prec \cA}  \theta^{U} \! (\cC,\cA) \\[1mm]
   & = \, \eta^{U} \! (\cA,\cA)
   \Bigl( \theta^{U} \! (\cA,\cA) \, -
   \sum_{\udo{\cC}\preccurlyeq \cA}  \theta^{U} \! (\cC,\cA)
   \Bigr)   \, = \, \eta^{U} \! (\cA,\cA) \, 
   \theta^{U} \! (\cA,\cA) \, = \, 1 \ts ,
\end{split}
\]
where the second step employs the induction hypothesis while the
second assertion of Proposition~\ref{prop:invertible} was used in the
penultimate step.

The third claim is a consequence of the initial conditions together
with Eqs.~\eqref{eq:def-b-coeff} and \eqref{eq:gen-summatory}, because
\[
   1 \, = \, b^{U}_{0} \! (\cA) \, = 
   \sum_{\udo{\cB} \succcurlyeq \cA}\eta^{U}\! (\cA,\cB) \, a^{U}_{0}
   \! (\cB) \, = \sum_{\udo{\cB} \succcurlyeq \cA}\eta^{U}\! 
   (\cA,\cB) \, \delta^{U}\! (\cB,\pmax) \, = \, \eta^{U} \! (\cA,\pmax)
\]
holds for any $\cA\in\PP (U)$.
\end{proof}

It is quite clear that the structure of $\eta$ is somewhat simpler
than that of $\theta$. At present, we have $\eta$ defined as the
inverse function to $\theta$ in the incidence algebra, but it would be
nice to also have a direct way to calculate it, for instance via
another recursion.

\section{Some comments on the singular cases}
\label{sec:degen}

Our focus in the previous section was on the generic case that
allowed for a recursively defined, general solution. Let us now
briefly look into what happens when certain degeneracies among
the decay rates occur.

Recall that our ansatz \eqref{eq:gen-ansatz} requires, a priori, the
linear independence of the exponential functions $\ee^{-\psi^{U} \!
  (\cA) \ts t}$, and thus the distinctness of the decay rates
$\psi^{U}\! (\cA)$, for $\cA\in\PP (U)$, separately for all non-empty
$U\nts \subseteq S$. A posteriori, we need to understand what happens
when two decay rates, as functions of the recombination rates
$\varrho^{S}$ of the system on the top level, become equal. Once
again, this is best looked at inductively.  If $\varnothing \ne U \nts
\subseteq S$ satisfies $\lvert U \rvert \leqslant 3$, we are in the
realm of the `linear' solution of Section~\ref{sec:linear}, and no
consequences emerge from degeneracies. In other words, the solution
formula from Proposition~\ref{prop:lin-solve} is valid for \emph{all}
values of the recombination rates, irrespective of possible
degeneracies; compare Corollaries~\ref{coro:simple-cases} and
\ref{coro:add-on}.

When we step up in system size, there can be degeneracies that are
still `harmless' in the sense that the $\theta$-coefficients extend
continuously to these situations. In this case, the (now reduced) set
of exponentials in Eq.~\eqref{eq:gen-ansatz} is still sufficient, and
the solution valid. However, as is evident from
Eq.~\eqref{eq:theta-rec}, the situation changes when $\psi^{U}\! (\cB)
= \psi^{U} \! (\pmax)$ for some $\cB\in\PP (U)$, as we then hit a
singularity. To understand the underlying phenomenon, let us assume
that $U$ corresponds to the smallest subsystem where this type of
`bad' degeneracy occurs. Now, rewrite Eq.~\eqref{eq:rec-reduce} as
\[
    \dot{a}^{\ts U}_{t} \! (\cA) \, = \, - \ts\ts \psi^{U} \! (\pmax)
    \, a^{\ts U}_{t}\! (\cA) \; + \! 
    \sum_{\cA \preccurlyeq \udo{\cB} \prec \pmax} \!
    \varepsilon (\cA,\cB) \, \ee^{-\psi^{U} \! (\cB) \ts t}
\]
with $ \varepsilon (\cA,\cB) = \sum_{\cB \preccurlyeq \udo{\cC} \prec
  \pmax} \varrho^{U} \! (\cC)\, \prod_{i=1}^{\lvert \cC\rvert}
\theta^{C_{i}} (\cA|^{\pa}_{\nts C_{i}} , \cB|^{\pa}_{\nts
  C_{i}})$. Note that the $ \varepsilon (\cA,\cB)$ are well-defined
for all $\cA,\cB\in\PP (U)$ and \emph{all} values of the recombination
rates, because only $\theta$-coefficients of subsystems smaller than
$U\!$ occur in the product. We are now in the standard situation of
ODE theory that we have summarised in Lemma~\ref{lem:help} of the
Appendix:\ Precisely when $\psi^{U}\! (\cB) = \psi^{U} \! (\pmax)$ for
some $\cB\in\PP (U)$, we need an additional function for our ansatz,
namely $t \cdot \ee^{-\psi^{U} \! (\cB) \ts t}$. So, our original
ansatz \eqref{eq:gen-ansatz} no longer suffices, and has to be
modified accordingly. At the same time, in line with our previous
observation, additional degeneracies of the form $\psi^{U}\! (\cB) =
\psi^{U} \! (\cB')$ with $\cB,\cB' \ne \pmax$ are harmless.

Now, this dichotomic structure continues on each new level: There are
`harmless' degeneracies that do not require additional functions for
our ansatz, while degeneracies of the type $\psi^{U}\! (\cB) =
\psi^{U} \! (\pmax)$ once again render the function set of the ansatz
incomplete, even one that was augmented in the previous
step. Lemma~\ref{lem:help2} of the Appendix reviews a typical case,
while the remarks following it show how the degenerate case produces
extra monomial factors of increasing exponents for each degeneracy of
the `bad' type. At this point, we hope that the general structure is
sufficiently clear, so that we can leave further details to the
reader.

Let us summarise our discussion as follows.
\begin{coro}
  If\/ $S$ is a finite set as before, the generic recursive solution
  from Theorem~$\ref{thm:rec-sol}$ extends to all recombination
  rates\/ $\varrho^{S} (\cA)$, with\/ $\cA\in\PP (S)$, such that\/
  $\psi^{U}\! (\pmax) \ne \psi^{U}\! (\cB)$ holds for all non-empty\/
  $U\nts\subseteq S$ and all\/ $\cB\in\PP (U)$.  \qed
\end{coro}

\section*{Appendix}

Here, we state some useful results from classical ODE theory, whose
proofs are straightforward exercises that are left to the reader.
\begin{lemma}\label{lem:help}
  Let\/ $\varrho$ and\/ $\sigma^{\pa}_{1}, \dots , \sigma^{\pa}_{m}$
  be non-negative real numbers.  Then, the Cauchy problem defined by
  the ODE
\[
     \dot{g} \, = \, -\varrho \, g \, + \sum_{i=1}^{m} 
     \varepsilon^{\pa}_{i} \ts \ee^{-\sigma^{\pa}_{i} \ts t}
\]   
together with the initial condition\/ $g(0) = g^{\pa}_{0}$ has the
unique solution
\[
     g(t) \, = \, g^{\pa}_{0} \, \ee^{-\varrho \ts t} \, +
     \sum_{i=1}^{m} \varepsilon^{\pa}_{i} \, 
      E^{\pa}_{0} (\varrho, \sigma^{\pa}_{i}; t) \ts ,
\]   
   where
\[
\begin{split}
    E^{\pa}_{0} (\alpha,\beta;t) \, & := \sum_{\ell=0}^{\infty} (-1)^{\ell} 
     \ts \bigl(\alpha^{\ell} + \alpha^{\ell-1} \beta + \dots +
      \alpha\ts \beta^{\ell-1} + \beta^{\ell} \bigr) \ts 
      \frac{t^{\ell+1}}{(\ell+1)!} \\
      & = \, \begin{cases}
      t \, \ee^{-\alpha\ts t} , & \text{if $\alpha = \beta$}, \\
      \frac{1}{\alpha - \beta}\, \bigl( \ee^{-\beta \ts t} - 
      \ee^{-\alpha\ts t} \bigr) ,
      & \text{otherwise} , \end{cases}
\end{split}
\]   
is a smooth function that is non-negative for all\/ $t\geqslant 0$ and
symmetric under the exchange of the two parameters\/ $\alpha$ and\/
$\beta$.  \qed
\end{lemma}

An illustration of the function $E^{\pa}_{0}$ is shown in the left
panel of Figure~\ref{fig:efun}.  A similar type of result emerges for
mixtures of exponentials with higher order monomials.

\begin{figure}
\begin{center}
  \includegraphics[width=0.87\textwidth]{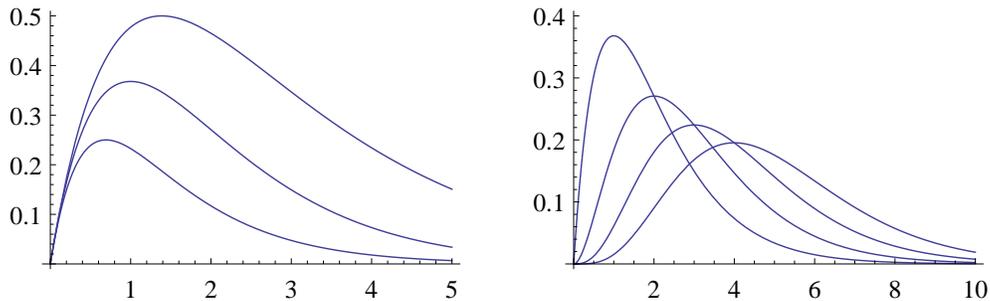}
\end{center}
\caption{Left panel:\ Illustration\label{fig:efun} of the function
  $E^{}_{0}(\alpha,\beta;t)$ of Lemma~\ref{lem:help} for $\alpha=1$
  and three values of $\beta$, namely $\beta=0.5$ (top curve),
  $\beta=1$ (middle) and $\beta=2$ (bottom). Right panel:\
  Illustration of $E_{m} (1,1;t)$ for parameters $m\in \{0,1,2,3\}$
  (from left to right).}
\end{figure}

\begin{lemma}\label{lem:help2}
  Let\/ $\varrho$ and\/ $\sigma$ be non-negative real numbers, and $m
  \in \NN_0$. Then, the Cauchy problem
\[
   \dot g \, = \, - \varrho\, g \, + \, 
   \varepsilon\ts \frac{t^m}{m!}\ts \ee^{-\sigma t}
\]
with initial condition\/ $g(0)=g_0$ has the unique solution\/
$g(t)=g^{\pa}_{0} \,\ee^{-\varrho t} + \varepsilon\, E_m(\varrho,
\sigma; t)$, where
\[
\begin{split}
    E_m(\varrho, \sigma; t) \, & := \,
     \ee^{-\varrho t} \int_0^t \frac{\tau^m}{m!} \,
     \ee^{(\varrho-\sigma) \tau} \dd \tau \\[1mm]
    & = \,
    \begin{cases}
       \ee^{-\varrho t}\ts \frac{(-1)^{m+1}}{(\varrho-\sigma)^{m+1}} +
       \ee^{-\sigma t} \sum_{\ell=0}^m \frac{(-1)^{\ell}\, t^{m-\ell} }
              {(m-\ell)!\, (\varrho - \sigma)^{\ell+1}} , &
       \text{if \;} \sigma \ne \varrho \ts , \\[2mm]
       \frac{t^{m+1}}{(m+1)!}\, \ee^{-\varrho t}, & \text{if \;} 
        \sigma = \varrho \ts ,
    \end{cases}
\end{split}
\]
is a smooth non-negative function for all $t \geqslant 0$. \qed
\end{lemma}

Some examples are illustrated in the right panel of
Figure~\ref{fig:efun}.  By standard results from ODE theory, it is
clear how to combine the results of Lemmas~\ref{lem:help} and
\ref{lem:help2} to cover further situations. Let us only add that the
ODE
\[
       \dot{g} \, = \, -\varrho\, g \, + \sum_{\ell = 0}^{m} 
       \varepsilon^{\pa}_{\ell} \,
       \frac{t^{\ell}}{\ell !} \, \ee^{-\varrho \ts t}
\]
with initial condition $g(0) = g^{\pa}_{0}$ has the unique solution
\[
       g(t) \, = \,  g^{\pa}_{0} \, \ee^{-\varrho\ts t} \, +  
       \sum_{\ell = 0}^{m} \varepsilon^{\pa}_{\ell} \,
       \frac{t^{\ell+1}}{(\ell + 1) !} \, \ee^{-\varrho \ts t} ,
\] 
which explains the appearance of monomial factors with increasing
powers in the solution of such equations with degenerate rates.

\section*{Acknowledgements}
MB would like to thank Roland Speicher for valuable discussions. We
thank an anonymous reviewer for his insightful and constructive
comments. This work was supported by the German Research Foundation
(DFG), within the SPP 1590.

\end{document}